\documentclass[10pt,a4paper]{amsart}
\usepackage[T1]{fontenc}
\usepackage[english]{babel}
\usepackage{amsthm,amssymb,amsfonts,mathrsfs,calligra,amsmath,fancyhdr}
\usepackage{enumitem}
\usepackage{scalerel}
\usepackage{rotating}
\usepackage{microtype}
\usepackage[maxbibnames=99,style=numeric, backend=biber, backref, backrefstyle=none, maxalphanames=6]{biblatex}
\newcommand{\q}[1]{``#1''}

\usepackage[usenames,dvipsnames,usenames,dvipsnames,svgnames,table]{xcolor}
\definecolor{newblue} {RGB} {0,100,200}
\usepackage[nodisplayskipstretch]{setspace}
\usepackage[unicode,pdftex, bookmarks, linktocpage=true]{hyperref}
\usepackage[babel]{csquotes}
\usepackage{microtype}
\usepackage[left=3cm, right=3cm,top=2.7cm,bottom=2.7cm]{geometry}
\usepackage[all]{xy}
\usepackage{bookmark}
\usepackage{orcidlink}
\usepackage{dutchcal}

\usepackage[unicode,pdftex, bookmarks, linktocpage=true]{hyperref}
\hypersetup{hypertexnames=false, colorlinks=true, citecolor=blue, linkcolor=blue, urlcolor=magenta, pdfpagemode=UseNone , breaklinks=true}

\usepackage{tikz}
\usetikzlibrary{matrix,arrows,calc,decorations,patterns,decorations.markings}
\usepackage{tikz-cd}

\usepackage{bm}

\usepackage{etoolbox}
\patchcmd{\section}{\normalfont}{\normalfont\normalsize}{}{}

\usepackage{bbm}
\usepackage{tkz-euclide}
\tikzset{
  dotted/.style={pattern=dots,pattern color=#1},
  dotted/.default=black
}
\tikzset{
  fdotted/.style={pattern=crosshatch dots,pattern color=#1},
  fdotted/.default=black
}
\tikzset{
  scopedlines/.style={pattern=north east lines,pattern color=#1},
  scopedlines/.default=black
}
\tikzset{
  hrlines/.style={pattern=horizontal lines,pattern color=#1},
  hrlines/.default=black
}
\newcommand*{\DashedArrow}[1][]{\mathbin{\tikz [baseline=-0.25ex,-latex, dashed,#1] \draw [#1] (0pt,0.5ex) -- (1.3em,0.5ex);}}
\usepackage{microtype}

\usepackage{mathtools}

\theoremstyle{plain}
\newtheorem{lem}{Lemma}[section]

\renewcommand\qedsymbol{$\blacksquare$}

\theoremstyle{definition}
\newtheorem{defn}[lem]{Definition}

\theoremstyle{definition}
\newtheorem{rmk}[lem]{Remark}

\theoremstyle{definition}

\theoremstyle{plain}
\newtheorem{cor}[lem]{Corollary}

\theoremstyle{plain}
\newtheorem{prop}[lem]{Proposition}

\theoremstyle{plain}
\newtheorem{thm}[lem]{Theorem}

\theoremstyle{plain}

\theoremstyle{definition}
\newtheorem{ex}[lem]{Example}

\newenvironment{proofA}{\vspace{0.2cm}\paragraph{\bf\textit{Proof  of Theorem \ref{thmA}}}}{\hfill \qedsymbol \medskip}

\newenvironment{proofB}{\vspace{0.2cm}\paragraph{\bf\textit{Proof  of Theorem \ref{thmB}}}}{\hfill \qedsymbol \medskip}

\newenvironment{proofC}{\vspace{0.2cm}\paragraph{\bf\textit{Proof  of Theorem \ref{thmC}}}}{\hfill \qedsymbol \medskip}

\newenvironment{proofD}{\vspace{0.2cm}\paragraph{\bf\textit{Proof  of Proposition \ref{propCY}}}}{\hfill \qedsymbol \medskip}

\newcommand{\Z}{\mathbf{Z}}
\newcommand{\C}{\mathbf{C}}

\renewcommand{\P}{\mathbf{P}}
\newcommand{\K}{\mathbf{K}}
\newcommand{\R}{\mathbf{R}}
\newcommand{\A}{\mathbf{A}}
\newcommand{\Q}{\mathbf{Q}}
\newcommand{\Cox}{\mathrm{Cox}}
\newcommand{\Nef}{\mathrm{Nef}}
\renewcommand{\dim}{\mathrm{dim}}

\renewcommand{\v}{\underline{v}}
\newcommand{\w}{\underline{w}}
\newcommand{\Pic}{\mathrm{Pic}}
\newcommand{\Cl}{\mathrm{Cl}}
\renewcommand{\O}{\mathscr{O}}
\DeclareMathOperator{\Hom}{\mathcal{H\!om}}

\newcommand\blfootnote[1]{%
  \begingroup
  \renewcommand\thefootnote{}\footnote{#1}%
  \addtocounter{footnote}{-1}%
  \endgroup
}

\usepackage{xparse}
\usepackage{environ}

\ExplSyntaxOn
\NewEnviron{pmatrixT}
 {
  \marine_transpose:V \BODY
 }

\int_new:N \l_marine_transpose_row_int
\int_new:N \l_marine_transpose_col_int
\seq_new:N \l_marine_transpose_rows_seq
\seq_new:N \l_marine_transpose_arow_seq
\prop_new:N \l_marine_transpose_matrix_prop
\tl_new:N \l_marine_transpose_last_tl
\tl_new:N \l_marine_transpose_body_tl

\cs_new_protected:Nn \marine_transpose:n
 {
  \seq_set_split:Nnn \l_marine_transpose_rows_seq { \\ } { #1 }
  \int_zero:N \l_marine_transpose_row_int
  \prop_clear:N \l_marine_transpose_matrix_prop
  \seq_map_inline:Nn \l_marine_transpose_rows_seq
   {
    \int_incr:N \l_marine_transpose_row_int
    \int_zero:N \l_marine_transpose_col_int
    \seq_set_split:Nnn \l_marine_transpose_arow_seq { & } { ##1 }
    \seq_map_inline:Nn \l_marine_transpose_arow_seq
     {
      \int_incr:N \l_marine_transpose_col_int
      \prop_put:Nxn \l_marine_transpose_matrix_prop
       {
        \int_to_arabic:n { \l_marine_transpose_row_int }
        ,
        \int_to_arabic:n { \l_marine_transpose_col_int }
       }
       { ####1 }
     }
   }
   \tl_clear:N \l_marine_transpose_body_tl
   \int_step_inline:nnnn { 1 } { 1 } { \l_marine_transpose_col_int }
    {
     \int_step_inline:nnnn { 1 } { 1 } { \l_marine_transpose_row_int }
      {
       \tl_put_right:Nx \l_marine_transpose_body_tl
        {
         \prop_item:Nn \l_marine_transpose_matrix_prop { ####1,##1 }
         \int_compare:nF { ####1 = \l_marine_transpose_row_int } { & }
        }
      }
     \tl_put_right:Nn \l_marine_transpose_body_tl { \\ }
    }
   \begin{pmatrix}
   \l_marine_transpose_body_tl
   \end{pmatrix}
 }
\cs_generate_variant:Nn \marine_transpose:n { V }
\cs_generate_variant:Nn \prop_put:Nnn { Nx }
\ExplSyntaxOff

\title[Birational geometry of hypersurfaces]{Birational geometry of hypersurfaces in products of weighted projective spaces}
\author{Francesco Antonio Denisi \orcidlink{0000-0002-1128-7890}}
\address{Université Paris Cité and Sorbonne Université, CNRS, IMJ-PRG, F-75013 Paris, France \\ \normalfont{E-mail: denisi@imj-prg.fr}}

\begin{document}
\begin{abstract}
We study the birational geometry of hypersurfaces in products of weighted projective spaces, extending results previously established by J.\ C.\ Ottem. For most cases where these hypersurfaces are Mori dream spaces, we determine all relevant cones and characterise their birational models, along with the small $\Q$-factorial modifications to them. We also provide a presentation of their Cox ring. Finally, we establish the birational form of the Kawamata-Morrison cone conjecture for terminal Calabi-Yau hypersurfaces in Gorenstein products of weighted projective spaces.

\end{abstract}
\maketitle
\blfootnote{This project has received funding from the European Research Council (ERC) under the European Union's Horizon 2020 research and innovation programme (ERC-2020-SyG-854361-HyperK).}

\section{Conventions and notations}
 Unless otherwise stated, we work over an uncountable, algebraically closed field $\K$ of characteristic 0. Throughout this paper, a variety will be a separated integral scheme of finite type over $\K$. Given a normal projective variety $X$ and a prime divisor $Y\subset X$, we will say that $Y$ is Cartier if the associated reflexive sheaf $\O_X(Y)$ is invertible.
\medskip

Given a polynomial ring $\K[x_0, \dots, x_n]$ with grading defined by $\mathrm{deg}(x_i) = v_i$, where $v_i$ is a positive natural number for each $i = 0, \dots, n$, we denote by $\P(\v)$ the projective $\K$-scheme $\mathrm{Proj}(\K[x_0, \dots, x_n]) \to \mathrm{Spec}(\K)$. The scheme $\P(\v)$ is called the \textit{weighted projective space} with weights $v_0, \dots, v_n$, and its dimension is $n$. Depending on whether we want (or need) to specify the dimension $n$ of a weighted projective space, we will write $\P^n(\v)$ or $\P(\v)$. Moreover, we set $a(\v) := \mathrm{l.c.m.}\{v_i\}_{i=0}^n$, i.e.\ the least common multiple of the weights. In this paper, all weighted projective spaces $\P(\v)$ are assumed to be well-formed; this means that the greatest common divisor of all weights is one, as well as the greatest common divisor of any subset of weights of cardinality $n$.
\medskip

Given a product $\P^n(\v) \times \P^m(\w)$, with projections $p_1$ and $p_2$, we denote by $\mathscr{O}_{\P^n(\v) \times \P^m(\w)}(d,e)$ (sometimes $\mathscr{O}(d,e)$) the reflexive sheaf
\[
(\mathscr{O}_{\P^n(\v)}(d) \boxtimes \mathscr{O}_{\P^m(\w)}(e))^{\vee\vee} = (p_1^*(\mathscr{O}_{\P^n(\v)}(d)) \otimes p_2^*(\mathscr{O}_{\P^m(\w)}(e)))^{\vee\vee}.
\]
To maintain consistency with Ottem's notation in \cite{Ottem15}, if $Y$ is a general Cartier hypersurface in $\P(\v) \times \P(\w)$, we denote $\mathscr{O}(1,0)_{|Y}$ by $H_1$ and $\mathscr{O}(0,1)_{|Y}$ by $H_2$. For any positive integer $k$, we denote $kH_1$ as $(p_1^*(\O_{\P(\v)}(k)))_{|Y}$, with analogous notation for $kH_2$. This notation will also apply to hypersurfaces in products of multiple weighted projective spaces. Sometimes, without specifying it, we will refer to $kH_1$ (resp.\ $kH_2$) as a Weil divisor for any positive integer $k$, with some abuse of notation.
\medskip

Even though a non-trivial effective line bundle on a weighted projective space $\P^n(\underline{v})$ may not be very ample (e.g.\ \cite[Remark 3, p.\ 138]{BR86}), we will denote by $\nu_e(\P^n(\underline{v}))$ the image of the Veronese-type morphism $\nu_e = \varphi_{|\O_{\P(\underline{w})}(e)|}$. Moreover, we denote by $I_e$ the homogeneous ideal of $\nu_e(\P^n(\underline{v}))$ in $\P^N$, where $N + 1 = h^0(\P(\underline{v}), \O_{\P(\underline{v})}(e))$.

\section{Introduction}
A normal and $\Q$-factorial projective variety $Y$ is called a Mori dream space if the following conditions are satisfied:
\begin{enumerate}
\item $h^1(Y,\mathscr{O}_Y)=0$,
\item  the nef cone $\Nef(Y)$ is generated by finitely many integral Cartier divisors which are semiample,
\item  there exist finitely many small $\Q$-factorial modifications $f_i\colon Y \DashedArrow Y'_i$ such that $Y'_i$ satisfies the first two items, for each $i$, and
\[\mathrm{Mov}(Y)=\bigcup_{i} f_i^*(\mathrm{Nef}(Y'_i)).\]
\end{enumerate}

Here, small $\Q$-factorial modifications are birational maps that do not change the underlying variety in codimension 1 (see Definition \ref{SQM} for a formal definition), $\mathrm{Mov}(Y)$ is the movable cone of $Y$ (see\ Definition \ref{movablecone}), and $\mathrm{Nef}(Y_i)$ is the nef cone of $Y_i$.
\vspace{0.1cm}

Mori dream spaces were introduced by Hu and Keel in \cite{HK00} and play an important role in birational geometry, as they exhibit ideal behaviour under the minimal model program (MMP). Indeed, it turns out that any MMP for any divisor can be carried out. Also, any nef divisor is semiample by definition. Due to these strict requirements, finding examples of Mori dream spaces is both interesting and challenging.
\medskip
 
In \cite{Ottem15}, Ottem studied the birational geometry of hypersurfaces in products of projective spaces and obtained a complete picture in the case of products of two projective spaces $X=\P^n\times \P^m$. Notably, he found that it is the bidegree $(d,e)$ that determines whether a hypersurface is a Mori dream space, and not its \q{ampleness}. In particular, he proved that, if the hypersurface has dimension at least 3, for $d > n$ and $e \geq 2$, the very general hypersurface $Y$ in the linear series $|\mathscr{O}_X(d, e)|$ is not a Mori dream space. Conversely, if $0<d \leq n$ or $d>0$ and $e=1$, the general hypersurface in $|\mathscr{O}_X(d, e)|$ is a Mori dream space, and he provided a presentation for its Cox ring. We refer the reader to Definition \ref{Coxring} for the definition of Cox ring, and to \cite[Proposition 1.2]{Ottem15} for an analogous result when $Y$ is a surface. We point out that Ito independently arrived at the same conclusion when $d\leq n$ and for $Y$ with dimension at least 3 (see \cite{Ito14}). However, Ito's results do not give information about the birational geometry of $Y$.
\medskip

Inspired by Ottem's work, this article aims to thoroughly investigate the birational geometry of Cartier hypersurfaces in products of two or more weighted projective spaces. We obtain a comprehensive understanding of the birational geometry of Cartier hypersurfaces \(Y\) of dimension at least 3 in a product \(\P(\v) \times \P(\w)$. When these hypersurfaces are Mori dream spaces, we compute all relevant cones in \(N^1(Y)_{\R}\) and explicitly describe all possible small \(\Q\)-factorial modifications \(Y \dashrightarrow Y'\) as well as the corresponding birational models \(Y'\). We also give a presentation for their Cox ring and, for specific values of \((d,e)\), identify the toric variety that governs the birational geometry of \(Y\), according to \cite[Proposition 2.11]{HK00}. As one may expect, also in this setting the bidegree $(d,e)$ determines whether the considered hypersurface is a Mori dream space. In particular, as it happens in the smooth case, for large values of $d$ the hypersurface is not a Mori dream space. The main result of the paper is the following.

\begin{thm}\label{thmA}
Let $X$ be a product $\P^n(\v) \times \P^m(\w)$, with $\dim(X)\geq 4$, and $\mathscr{O}_X(d,e)$ a line bundle on $X$, where $d$ and $e$ are positive integers. Furthermore, set $a(\w) := \mathrm{l.c.m.}\{w_i\}_{i=0}^m$.
\begin{enumerate}
\item If $n,m \geq 2$, the general member $Y$ of $|\mathscr{O}_X(d,e)|$ is a Mori dream space, and 
\[
\Cox(Y)=\K [x_0,\dots,x_n,y_0,\dots,y_m]/(f),
\]
where $f$ is an equation for $Y$ in $\P^n(\v) \times \P^m(\w)$.

\item If $X=\P^1 \times \P^m(\w)$, with $1 \leq d \leq m$, and $k\geq 1$, the general member $Y$ of $|\mathscr{O}_X(d,ka(\w))|$ is a Mori dream space.
\vspace{0.2cm} 

$\bullet$ If $1<d<m$, we have
\[
\mathrm{Eff}(Y)=\mathrm{Mov}(Y)=\R_{\geq 0}H_1+\R_{\geq 0}(ka(\w)H_2-H_1),
\] 
\[
\mathrm{Nef}(Y)=\R_{\geq 0}H_1+\R_{\geq 0}H_2.
\]
In this case, $Y$ admits a small $\Q$-factorial modification, a fibration onto $\P^1$ and a dominant rational fibration to $\P^{d-1}$.

$\bullet$ If $d=1$, we have 
\[
\mathrm{Eff}(Y)=\R_{\geq 0}H_1+\R_{\geq 0}(ka(\w)H_2-H_1),
\]
\[
\mathrm{Nef}(Y)=\mathrm{Mov}(Y)=\R_{\geq 0}H_1+\R_{\geq 0}H_2.
\]
In this case, $Y$ admits a divisorial contraction and a fibration onto $\P^1$.
\vspace{0.2cm}

$\bullet$ If $d=m$, we have
\[
\mathrm{Eff}(Y)=\mathrm{Mov}(Y)=\mathrm{Nef}(Y)=\R_{\geq 0}H_1+\R_{\geq 0}(ka(\w)H_2-H_1).
\] 
Moreover, if $$Y=\{x_0^df_0+x_0^{d-1}x_1f_1+\dots+x_0x_1^{d-1}f_{d-1}+x_1^d=0\},$$ a presentation for the Cox ring of $Y$ is
\[\mathrm{Cox}(Y)\cong \K[x_0,x_1,y_0,\dots,y_m,z_1,\dots,z_d]/I,\]
where $I=(f_0+x_0z_1,f_1-x_0z_1+x_1z_2,\dots,f_{d-1}-x_0z_{d-1}+x_1z_d,f_d-x_0z_d)$.

In this case, $Y$ has a fibration onto $\P^{d-1}$ and a fibration onto $\P^1$.

 \item If $X=\P^1\times \P^m(\w)$, with  $d \geq m+1$, and $k\geq 2$, the very general member of $|\mathscr{O}_X(d,ka(\w))|$ is not a Mori dream space. In particular $\mathrm{Eff}(Y)\subsetneq \overline{\mathrm{Eff}}(Y)$.
 \item If $X=\P^1 \times \P^m(\w)$, with $d>m$ and $k=1$, the special member of $|\mathscr{O}_X(d,a(\w))|$ is a Mori dream space.
\end{enumerate}
\end{thm}

If in Theorem \ref{thmA} the factors of  $\P^n(\v) \times \P^m(\w)$ are smooth, we recover part of the statement of \cite[Theorem 1.1]{Ottem15}. Also, in item (2), we notice that, as in the smooth case, we could conclude that the general member of $|\mathscr{O}_X(d,e)|$ is a Mori dream space by using a criterion by Ito (see\ \cite{Ito14}). However, we aim to describe explicitly the birational geometry of the general member of $|\mathscr{O}_X(d,e)|$, and the criterion of Ito does not provide information about that. It is worth noting that all the considered hypersurfaces have klt singularities.
\vspace{0.2cm}

We also have an analogous statement to Theorem \ref{thmA} for surfaces in $\P^1\times \P^2(\w)$, which parallels \cite[Proposition 1.2]{Ottem15}.

\begin{thm}\label{thmB}
Let $X$ be a product $\P^1 \times \P^2(\w)$. Suppose that $X$ is Gorenstein, so that by \cite[Section 2]{BMP11} we have either $X=\P^1 \times \P(1,1,2)$, or $X=\P^1 \times \P(1,2,3)$. Let $k$ be any positive integer.
\begin{enumerate}
\item If $X=\P^1 \times \P(1,1,2)$, for $d \geq 3$ the very general member of $|\mathscr{O}_X(d,2k)|$ is not a Mori dream space. Instead, the very general member $Y$ of $|\mathscr{O}_X(1,2)|$ and $|\mathscr{O}_X(2,2k)|$ is a Mori dream space. Moreover, if $Y \in |\mathscr{O}_X(2,2k)|$ and $Y=\{x_0^2f_0+x_0x_1f_1+x_1^2f_2=0\}$, we have
\[
\mathrm{Eff}(Y)=\mathrm{Nef}(Y)=\R_{\geq 0}H_1+\R_{\geq 0}(eH_2-H_1),
\] 
and
$$\mathrm{Cox}(Y)\cong \K[x_0,x_1,y_0,y_1,y_2,z_1,z_2]/I,$$ where $I=(f_0+x_0z_1,f_1-x_0z_1+x_1z_2,f_2-x_0z_2)$. 

\item If $X=\P^1 \times \P(1,2,3)$, the very general member of $|\mathscr{O}_{X}(d,6k)|$, with $d \geq 3$, is not a Mori dream space. Instead, the very general member of $|\mathscr{O}_{X}(2,6k)|$ is a Mori dream space. Moreover, if we set $Y=\{x_0^2f_0+x_0x_1f_1+x_1^2f_2=0\}$, we have
\[
\mathrm{Eff}(Y)=\mathrm{Nef}(Y)=\R_{\geq 0}H_1+\R_{\geq 0}(eH_2-H_1)
\] 
 and $$\mathrm{Cox}(Y)\cong \K[x_0,x_1,y_0,y_1,y_2,z_1,z_2]/I,$$
 where $I=(f_0+x_0z_1,f_1-x_0z_1+x_1z_2,f_2-x_0z_2)$. 
\end{enumerate}
\end{thm}

Note that in Theorem \ref{thmB} there are cases where we could not make any statements about the general, very general, or even special member of the considered linear series. Specifically, for $X=\P(1,1,2)$, the case $|\mathscr{O}_{X}(1,4k)|$ remains open and for $X=\P(1,2,3)$, the case $|\mathscr{O}_{X}(1,6k)|$ is still unresolved, for any $k > 0$. In both instances, it seems that the very general member is not a Mori dream space when $k$ is sufficiently large, similar to the behaviour observed for the blow-up of the projective plane at $k\geq 9$ points in general position.
\medskip

In the proofs of Theorems \ref{thmA}, \ref{thmB}, we use degeneracy loci of morphisms of vector bundles to establish the existence (and the uniqueness) of the corresponding small $\Q$-factorial modification. We make use of results by Okawa in \cite{Okawa} and by Ottem in \cite{Ottem15} to compute the various Cox rings, and further apply Ottem's results to show that for large values of $d$, the very general member of $|\O_{\P^1\times \P(\w)}(d,e)|$ is not a Mori dream space.
\medskip

As it happens in the smooth case (see Ottem's \cite[Proposition 6.1]{Ottem15}), we expect that the situation becomes more complicated when we consider hypersurfaces in products of more than two weighted projective spaces. 

\begin{prop}\label{propCY}
Suppose that $$X=(\P^1)^{n}\times \P^{m_1}(\w^1)\times \cdots \times \P^{m_k}(\w^k)\to \mathrm{Spec}(\C)$$ is Gorenstein of dimension $\mathrm{dim}(X)\geq 4$, where $m_i\geq 2$ for any $i=1,\dots,k$. Then the general member $Y\in |-K_X|$ is a Calabi-Yau hypersurface and we have $$\mathrm{Nef}(Y)=\R_{\geq 0}H_1+\R_{\geq 0}H_2+\cdots+\R_{\geq 0}H_{n+k}.$$ Moreover, the following hold:
\begin{enumerate}
\item If $n=0$, then $\mathrm{Eff}(Y)=\mathrm{Mov}(Y)=\mathrm{Nef}(Y)=\mathrm{Eff}(X)_{|Y}$ and $Y$ is a Mori dream space.
\item If $n=1$, the effective cone of $X$ strictly contains that of $Y$, and also in this case $Y$ is a Mori dream space.
\item If $n>1$, then $Y$ is not a Mori dream space. In particular, $\mathrm{Bir}(Y)$ is infinite. 
\end{enumerate}
\end{prop}

Still, we show that under the terminality assumption, the Calabi-Yau hypersurfaces considered in Proposition \ref{propCY} satisfy the birational form of the Kawamata-Morrison cone conjecture. Recall that a normal and $\Q$-factorial complex projective variety $X$ is called Calabi-Yau if $K_X \equiv 0$, and $h^1(X,\mathscr{O}_X)=0$. The (birational) Kawamata-Morrison cone conjecture for a terminal Calabi-Yau variety $X$ predicts the existence of a fundamental domain for the action of $\mathrm{Bir}(X)$ on $\overline{\text{Mov}}^+(X)$. In essence, the conjecture posits that $\overline{\text{Mov}}^+(X)$ is potentially rational polyhedral. If this is not the case, the cause lies with the group \(\mathrm{Bir}(X)\) of birational self-maps of $X$, which, in such circumstances, is infinite. The infinite nature of the group action results in the fundamental domain being shifted across the entire cone $\overline{\text{Mov}}^+(X)$, leading to a more intricate and non-polyhedral structure. We refer the reader to Section \ref{preliminaries} for the definition of $\overline{\text{Mov}}^+(X)$.

\begin{thm}\label{thmC}
Let $X$ be the complex projective variety $(\P^1)^n \times \left(\prod_{i=1}^k \P^{m_i}(\w^i)\right) \to \mathrm{Spec}(\C)$, with $n\geq 0$. Suppose that $X$ is Gorenstein and terminal. Then the general member $Y$ of  $|-K_X|$ verifies the birational form of the Kawamata-Morrison cone conjecture. Moreover, $\overline{\mathrm{Mov}}^+(Y)=\mathrm{Mov}(Y)$.
\end{thm}

We point out that one can construct terminal products of weighted projective spaces following the work \cite{Kasp13} of Kasprzyk.
\vspace{0.2cm}

The Kawamata-Morrison cone conjecture for Calabi-Yau hypersurfaces, or Calabi-Yau complete intersections in products of projective spaces, has been investigated by several people. Our theorem provides a small contribution in the singular case to this. Oguiso and Cantat proved in \cite{CO15} the birational form of the cone conjecture for general anti-canonical hypersurfaces in $(\P^1)^M$, for $M\geq 4$. In his Master's thesis \cite{Skauli17}, Skauli generalised the result of Oguiso and Cantat to certain general Calabi-Yau complete intersections in $(\P^n)^M$. To conclude, in \cite{Yanez22}, Yáñez further generalised the result of Cantat and Oguiso (and also that of Skauli) to some general Calabi-Yau complete intersections in products of projective spaces of possibly different dimensions.  

\section*{Organisation of the paper}
$\bullet$ In Section \ref{preliminaries} we gather all the necessary notions and results that will be used throughout the paper.

$\bullet$ In Section \ref{hypersurfaces}, we study the birational geometry of Cartier hypersurfaces in products of the form $\P^n(\v) \times \P^m(\w)$. In particular, we prove Theorem \ref{thmA} and Theorem \ref{thmB}. Subsection \ref{toricvariety} provides a detailed description of the birational geometry of the toric variety that governs certain hypersurfaces in products of the type $\P^1 \times \P(\w)$.

$\bullet$ Section \ref{morespaces} addresses the case of hypersurfaces in products of more than two weighted projective spaces. We observe that even hypersurfaces with low multidegree may fail to be Mori dream spaces. Some of these examples include Calabi-Yau anticanonical hypersurfaces, for which we verify the birational version of the Kawamata-Morrison cone conjecture.

\section*{Acknowledgments}
I thank J.\ C.\ Ottem for answering questions via email and A.\ Massarenti for useful discussions and suggestions. I also thank V.\ Benedetti, E.\ Fatighenti, S.\ Filipazzi, A.\ Petracci, F.\ Polizzi, and N.\ Tsakanikas for additional helpful discussions. Finally, I thank E.\ Floris for comments on a preliminary version of the article, and F.\ Giovenzana and F.\ Rizzo for the interest they took in this work.

\section{Preliminaries}\label{preliminaries}

Let $X$ be a normal projective variety of dimension $d$ over $\K$. The Néron-Severi space of $X$ is the finite-dimensional vector space $N^1(X)_{\R}:=\mathrm{Div}(X)\otimes \R/\equiv$, where $\mathrm{Div}(X)$ is the group of integral Cartier divisors of $X$, and $\equiv$ is the numerical equivalence relation. The elements of $\mathrm{Div}(X)\otimes \R$ (resp. $\mathrm{Div}(X)\otimes \Q$) are called Cartier $\R$-divisors (resp.\ $\Q$-divisors). Two Cartier $\R$-divisors $D,D'$ are numerically equivalent if $D\cdot C=D'\cdot C$ for any irreducible curve $C$ of $X$.
\medskip

The class group $\mathrm{Cl}(X)$ of $X$ is the group $\mathrm{WDiv}(X)$ of Weil divisors on $X$, modulo the linear equivalence relation. 
\medskip

\begin{defn}\label{Coxring}
Let $X$ be a normal and $\Q$-factorial projective variety. Suppose that the class group $\mathrm{Cl}(X)$ of $X$ is free and finitely generated. Moreover, let $\Gamma$ be a subgroup of $\text{WDiv}(X)$, such that the quotient map $\Gamma\to \mathrm{Cl}(X)$ is an isomorphism. We define the Cox ring of $X$ as
\[
\mathrm{Cox}(X):=\bigoplus_{D\in \Gamma}H^0(X,\O_X(D)),
\] 
where $\O_X(D)$ denotes the reflexive sheaf associated with the Weil divisor $D$. The ring structure is given by the tensor product of global sections. Moreover, there is a natural grading on $\mathrm{Cox}(X)$, induced by $\Gamma$, as follows: an element of  $H^0(X,\O_X(D))$ has degree $D$.
\end{defn}

The Cox ring of a normal projective variety encodes a lot of information about the variety, and under certain assumptions detects whether the considered variety is a Mori dream space. Indeed, according to \cite[Proposition 2.9]{HK00}, a normal and $\Q$-factorial projective variety $X$ is a Mori dream space if and only if its Cox ring is a finitely generated $\K$-algebra. For an excellent account of Cox rings we refer the reader to \cite{ADHL14}.

\begin{defn}\label{SQM}
A small $\Q$-factorial modification of a normal and $\Q$-factorial projective variety $X$ is a birational map $f\colon X \DashedArrow X'$ to another normal and $\Q$-factorial projective variety $X'$ that is an isomorphism in codimension $1$. This means that $f$ (resp.\ $f^{-1}$) restricts to an isomorphism $f \colon U \to U'$ (resp.\ $f^{-1} \colon U' \to U$), where $U$ and $U'$ are open subsets whose complements have codimension greater than or equal to $2$.
\end{defn}

\begin{defn} Let $L$ be a line bundle on $X$. For any $k$, let $\phi_k = \phi_{|L^{\otimes k}|} \colon X \DashedArrow \P(H^0(X,L))$ be the rational map induced by $L^{\otimes k}$, and denote by $\mathrm{im}(\phi_k)$ the image of $\phi_k$, that is, the image of the closure of $\Gamma(\phi_k) \subset X \times  \P(H^0(X,L))$ via the projection $ X \times  \P(H^0(X,L)) \to \P(H^0(X,L))$. We define the Itaka dimension of $L$ as
\[
\kappa(L):=\sup\left\{\dim(\mathrm{im}(\phi_k))\;|\;k>0 \text{ and } h^0(X,L^{\otimes k})>0\right\}.
\]
If $D$ is any integral Cartier divisor on $X$, we set $\kappa(D):=\kappa(\mathscr{O}_X(D))$.
\end{defn}

An integral Cartier divisor $D$ on $X$ is big if $\kappa(\mathscr{O}_X(D))=d$, i.e.\ if its Itaka dimension is maximal. Equivalently, there exists a constant $C$ such that $h^0(X,\mathscr{O}_X(mD))\geq C\cdot m^d$, for every $m\gg 0$ such that $h^0(X,\mathscr{O}_X(mD))\neq 0$.
\medskip

Now, let $V$ be a finite-dimensional $\R$-vector space. A cone in $V$ is a subset $K$ such that if $v\in K$, also $\lambda \cdot v$ belongs to $K$, where $\lambda$ is any positive real number. We say that $K$ is convex if it is convex as a subset of $V$. A convex cone is polyhedral if it is spanned by finitely many vectors in $V$. A rational polyhedral cone in $V$ is a convex cone spanned by finitely many rational vectors.
\medskip
 

The effective cone of $X$ is defined as the cone spanned by effective Cartier divisor classes and is denoted by $\mathrm{Eff}(X)$. The pseudo-effective cone is the closure of the effective cone in $N^1(X)_{\R}$ and is denoted by $\overline{\mathrm{Eff}}(X)$. The nef cone $\mathrm{Nef}(X)$ of $X$ is the cone in the Néron-Severi space of divisor classes $\alpha$ such that $\alpha\cdot C\geq 0$, for any integral curve $C$ in $X$. For an excellent account of divisors and cones in the Néron-Severi space, we refer the reader to \cite{Laz}.

\begin{defn}\label{movablecone}
Let $X$ be a normal projective variety. The movable cone $\text{Mov}(X)$ of $X$ is the cone in the Néron-Severi space spanned by movable divisor classes, i.e.\ the classes of integral Cartier divisors whose stable base locus has codimension greater than or equal to 2. Moreover, we define $\overline{\mathrm{Mov}}^{+}(X)$ as the convex hull of $\overline{\mathrm{Mov}}(X) \cap N^1(X)_{\Q}$ in $N^1(X)_{\R}$.
\end{defn}

 Let $X$ be a terminal Calabi-Yau variety, and $\mathrm{Bir}(X)$ the group of its birational self-maps. Consider the representation map $\rho \colon \mathrm{Bir}(X)\to \mathrm{GL}(N^1(X)_{\R})$. We say that a rational polyhedral cone $\Pi \subset \overline{\text{Mov}}^+(X)$ is a fundamental domain for the action of $\mathrm{Bir}(X)$ on $\overline{\mathrm{Mov}}^+(X)$ if 
 $$\mathrm{Bir}(X)\cdot \Pi:=\bigcup_{g \in \rho(\mathrm{Bir(X)})}g(\Pi)=\overline{\mathrm{Mov}}^+(X),$$ 
 and for any $g\in \rho(\mathrm{Bir}(X))$, we have $\mathrm{int}(\Pi)\cap \mathrm{int}(g(\Pi))=\emptyset$, unless $g$ is the identity.
\medskip

\begin{defn}
Let $X$ be a projective variety and $\mathscr{E}$ a vector bundle (i.e.\ a locally free $\mathscr{O}_X$-module of constant rank). We define $V(\mathscr{E}):=\mathcal{S\!pec}_{\O_X}(\mathrm{Sym}(\mathscr{E}^{\vee}))$, i.e.\ $V(\mathscr{E})\to X$ is the geometric vector bundle associated with $\mathscr{E}$. 
\end{defn}

\begin{defn}
Let $X$ be a projective variety, and $\varphi \colon \mathscr{E} \to \mathscr{F}$ a morphism of vector bundles on $X$, with $\mathrm{rk}(\mathscr{E})\geq \mathrm{rk}(\mathscr{F})$. We know that $\varphi \in H^0(X,\mathscr{E}^{\vee}\otimes \mathscr{F})$. Let $k$ be a positive integer, and consider 
\[
\wedge^k \varphi\in H^0(\wedge^k{\mathscr{E}^{\vee}}\otimes \wedge^k{\mathscr{F}}).
\]
The section $\wedge^k \varphi$ defines a morphism of vector bundles 
\[
h \colon \mathscr{O}_X \to \wedge^k{\mathscr{E}^{\vee}}\otimes \wedge^k{\mathscr{F}}.
\]
 If we dualise $h$, we obtain 
 \[
 h^{\vee} \colon (\wedge^k{\mathscr{E}^{\vee}}\otimes \wedge^k{\mathscr{F}})^{\vee} \to \mathscr{O}_X.
 \]
 Then we define the $k$-th degeneracy locus of $\varphi$ as the closed subscheme $D_k(\varphi)$ defined by the quasi-coherent sheaf of ideals $\mathscr{I}=\mathrm{im}(h^{\vee})$. Set-theoretically, $D_k(\varphi)$ consists of the (closed) points $x$ of $X$ where the induced morphism $\varphi_x \colon \mathscr{E}_x/\mathfrak{m}_x \mathscr{E}_x \to \mathscr{F}_x/\mathfrak{m}_x \mathscr{F}_x$ between the fibres has rank smaller than or equal to $k$.
\end{defn}

The following proposition tells us that the Itaka dimension of line bundles on normal projective varieties behaves well under surjective morphisms of normal projective varieties.

\begin{prop}[Proposition 1.5, \cite{Mori87}]\label{prop:Itakadim}
Let $f \colon X \to Y$ be a surjective morphism between normal projective varieties and $L$ a line bundle on $Y$. Then $\kappa(Y,L)=\kappa(X,f^*L)$.
\end{prop}



Recall that a Noetherian scheme $X$ is Cohen-Macaulay if for any closed point $x$ the ring $\O_{X,x}$ is Cohen-Macaulay, i.e.\ $\mathrm{depth}(\mathscr{O}_{X,x})=\mathrm{dim}(\mathscr{O}_{X,x})$.
We say that an embedded projective variety $X\subset \P^n$ is arithmetically Cohen-Macaulay if its homogeneous coordinate ring is Cohen-Macaulay. Equivalently, the affine cone over $X$ in $\A^{n+1}$ is a Cohen-Macaulay variety. An arithmetically Cohen-Macaulay variety is Cohen-Macaulay, but the vice versa does not hold in general. All weighted projective spaces are normal, Cohen-Macaulay and $\Q$-factorial varieties. The following lemma will be useful for our purposes.

\begin{lem}\label{lem:CM}Let $f\colon X\to Y$ be a flat morphism of varieties, with $Y$ Cohen-Macaulay. Suppose that the scheme-theoretic fibres $X_y$, with $y=f(x)$ and $x$ a closed point are Cohen-Macaulay. Then the scheme $X$ is Cohen-Macaulay.
\end{lem}
\begin{proof}
It suffices to show the statement for the closed points of $X$. Let $x$ be a closed point of $X$. We want to show that $\O_{X,x}$ is Cohen-Macaulay. The point $x$ is closed in $X_y$, and the same is true for the point $y$ of $Y$. Since $Y$ and $X_y$ are Cohen-Macaulay, the statement follows by \cite[Lemma 10.163.3]{SP24}.
\end{proof}

If $M^{m\times n}(\K)\cong \A^{mn}$ is the affine space of matrices of size $m\times n$ over $\K$, we denote by $M_k$ the affine closed subscheme of $\A^{mn}$ parametrising the matrices $M$ of rank $\mathrm{rk}(M)\leq k$. The scheme $M_k$ is a Cohen-Macaulay affine variety (see for example \cite[ (3.1) Theorem]{ACGH13}).
\medskip

The lemma below can be found in the form of speculation for smooth varieties in \cite[p.\ 590]{Banica91}, without a proof. For the reader's convenience, we state it in the Cohen-Macaulay setting and provide proof.

\begin{lem}\label{lem:CMvar}
Let $X$ be a Cohen-Macaulay variety, $E$, $F$, two vector bundles, and suppose that the vector bundle $\Hom(E,F)$ is globally generated. Let $s_0\colon X\hookrightarrow V(\Hom(\wedge^kE,\wedge^kF))$ be the zero section of  the vector bundle $V(\Hom(\wedge^kE,\wedge^kF))\to X$, and $k$ any positive integer. Furthermore, let $\Sigma$ be the inverse image scheme of $s_0(X)$ via the natural morphism $V(\Hom(E,F))\to V(\Hom(\wedge^kE,\wedge^kF))$, and denote by $V$ the inverse image scheme of $\Sigma$ via the natural morphism of vector bundles $X\times H^0(E^{\vee}\otimes F) \to V(\Hom(E,F))$. Then $V$ and $\Sigma$ are Cohen-Macaulay varieties.
\end{lem}
\begin{proof} 
We start by proving that $\Sigma$ is a Cohen-Macaulay subvariety of $V(\Hom(E,F))$. Indeed, the morphism $\Sigma \to X$ is a locally trivial fibration, locally given by $U_i \times M_k \to U_i$, for some suitable open subsets $U_i$ covering $X$, and hence is flat. This also proves that $\Sigma$ is integral. The fibres of this fibration are all isomorphic to $M_k$, which we recall being a Cohen-Macaulay affine variety. As $X$ is Cohen-Macaulay, we can use Lemma \ref{lem:CM} and deduce that $\Sigma$ is Cohen-Macaulay. Now, the morphism $X\times H^0(E^{\vee}\otimes F) \to V(\Hom(E,F))$ of vector bundles is surjective, since $\Hom(E,F)$ is globally generated. Also, the morphism $V \to \Sigma$ is flat, and the fibres are isomorphic to an affine space of a certain fixed dimension. It follows that $V$ is integral, and also in this case we can use Lemma \ref{lem:CM} to conclude that $V$ is a Cohen-Macaulay variety.
\end{proof}

To conclude this section, we recall how to pull back Weil divisors under finite morphisms. Let $f\colon W \to Z$ be a finite surjective morphism of normal projective varieties. Let $D$ be a Weil divisor on $Y$. With some abuse of notation, we also denote by $f$ the restricted morphism $f_{|f^{-1}(Z_{\mathrm{reg}})}$, where $Z_{\mathrm{reg}}$ is the regular locus of $Z$. Then the pull-back $f^*(D)$ is defined as $\overline{f^*(D_{|Z_{\mathrm{reg}}})}$, where the closure denotes the unique Weil divisor on $W$ extending the Cartier divisor $f^*(D_{|Z_{\mathrm{reg}}})\subset f^{-1}(Z_{\mathrm{reg}})$. This operation respects the linear equivalence relation among Weil divisors, and hence, when $D$ is effective, it induces an injective morphism of spaces of global sections $H^0(Z,\O_Z(D))\to H^0(W,\O_W(f^*(D)))$, where  $\O_Z(D)$ and $\O_W(f^*(D))$ are the reflexive sheaves associated with the Weil divisors $D$ and $f^*(D)$ respectively.
\medskip

We refer the reader to \cite{Schwede} for an account on Weil divisors and reflexive sheaves.

\section{Hypersurfaces in products of two weighted projective spaces}\label{hypersurfaces}
Recall that since weighted projective spaces are connected and have vanishing irregularity, we have
\[
\Pic(\P^n(\v) \times \P^m(\w)) \cong \Pic(\P^n(\v)) \times \Pic(\P^n(\w)) \cong \Z \oplus \Z.
\] 
In particular, it is generated by $\{\O(a(\v),0),\O(0,a(\w))\}$. Instead, the class group of $\P(\v) \times \P(\w)$  is a finitely generated, free abelian group, with a basis given by the two reflexive sheaves $\O(1,0),\O(0,1)$. If $x_0,\dots,x_n$ and $y_0,\dots,y_m$ are the coordinates of $\P(\v)$ and $\P(\w)$ respectively, then $$\mathrm{Cox}(\P(\v) \times \P(\w))=\K[x_0,\dots,x_n,y_0,\dots,y_m].$$

 We start by studying the situation where the factors appearing in $\P(\v) \times \P(\w)$ have dimensions at least 2.

\begin{prop}\label{prop1}
Consider a product $X=\P^n(\v) \times \P^m(\w)$, with $\dim(X)\geq 4$. Consider a line bundle $\O(d,e)$, with $d$ and $e$ positive numbers. Then the general member $Y$ of $|\O(d,e)|$ has
\[
\Cox(Y)\cong \K [x_0,\dots,x_n,y_0,\dots,y_m]/(f),
\]
where $f$ is an equation defining $Y$. In particular, $Y$ is a Mori dream space.
\end{prop}
\begin{proof}
Since $d$ and $e$ are positive numbers, the line bundle $\mathscr{O}(Y)\cong \mathscr{O}(d,e)$ is base point-free and ample, and hence, since $Y$ is general, we can apply the Grothendieck-Lefschetz Theorem (see\ \cite[Theorem 1]{RS06}) and conclude that $\Cl(Y) \cong \Cl(\P(\v)) \times \Cl(\P(\w))$. Then a basis for $\Cl(Y)$ (resp.\ $\mathrm{Pic}(Y)$) is given by the restriction of $\O(1,0)$ and $\O(0,1)$ (resp.\ $\O(a(\v),0)$ and $\O(0,a(\w))$) to $Y$. 



We now compute the Cox ring of $Y$. Consider the quotient morphism $$\pi\colon X'=\P^n\times \P^m\to \P(\v)\times \P(\w)=X.$$ The general member $Y$ of $|\O_X(d,e)|$ is pulled back to a smooth member $Y'$ of $|\O_{X'}(d,e)|$, of Picard number 2. Note that $Y'=\pi^*(Y)$ coincides with the inverse image scheme $X'\times_X Y$ via $\pi$ (see\ \cite[Proposition 11.48, Corollary 11.49]{GW10}). Let $x_0',\dots,x_n'$ (resp.\ $y_0',\dots,y_m'$) be the coordinates of $\P^n$ (resp.\ $\P^m$). Using the short exact sequence associated with $Y'$, and the vanishing of cohomology in degree $1$, we see that the Cox ring of $Y'$ is $\mathrm{Cox}(Y')=\K[x_0',\dots,x_n',y_0',\dots,y_m']/(f')$, where $f'=\pi^*(f)$ is an equation for $Y'$, and $f$ is an equation for $Y$. Consider the ring $R=\K[x_0,\dots,x_n,y_0,\dots,y_m]/(f)$. We have a natural ring homomorphism $R\to \mathrm{Cox}(Y')$, defined by $x_i\mapsto (x_i')^{v_i}$, $y_j\mapsto (y_j')^{{w_j}}$, and which factors through $\mathrm{Cox}(Y)$. The image of the injective ring homomorphism $\mathrm{Cox}(Y)\hookrightarrow \mathrm{Cox}(Y')$, obtained via pull back of Weil divisors, is the subring $$\K\left[(x_0')^{v_0},\dots,(x_n')^{v_n},(y_0')^{w_0},\dots,(y_m')^{w_m}\right]/(f').$$ Recall that the Cox ring of $Y$ has Krull dimension equal to $\mathrm{dim}(Y)+\rho(X)$, by \cite[Proposition 2.9]{HK00}, where $\rho(X)$ is the Picard number of $X$. But then $R\to \mathrm{Cox}(Y)$ is onto, and the rings $R,\mathrm{Cox}(Y)$ are integral domains with the same Krull dimension. This implies that $R\cong \mathrm{Cox}(Y)$.
\end{proof}

\begin{rmk}
We observe that if $X=\P^n(\v)\times \P^m(\w)$, with $n,m\geq 4$, the same argument in the proposition above applies to general Cartier hypersurfaces in the linear series $|\O(d,0)|$, $|\O(0,e)|$. Indeed, if $V$ is any general Cartier hypersurface in $|\O_{\P^n(\v)}(d)|$, it is sufficient to observe that $$\mathrm{Cl}(V\times \P(\w))\cong \mathrm{Pic}(V_{\mathrm{reg}}\times \P(\w)_{\mathrm{reg}})\cong  \mathrm{Pic}(V_{\mathrm{reg}})\times \mathrm{Pic}(\P(\w)_{\mathrm{reg}})\cong \mathrm{Cl}(V)\times \mathrm{Cl}(\P(\w)),$$ where the second isomorphism is obtained using \cite[Satz 1.7]{Friedrich74}. Then, under our assumptions, the Class group of $V\times \P(\w)$ has rank 2, and so we can argue as in the proposition above to show that $V\times \P(\w)$ is a Mori dream space. The same reasoning applies for hypersurfaces of the type $\P(\v)\times W$, where $W$ is a general Cartier hypersurface in $|\O_{\P(\w)}(e)|$.
\end{rmk}

The proposition above implies that, as in the smooth case, the only products (with two factors) where we can find non-Mori dream Cartier ample hypersurfaces are those of the type $\P^1 \times \P^m(\w)$, with $m\geq 2$. 

\begin{lem}\label{lem:reduced}
Consider a product $X=\P^{d-1} \times \P^m(\w)$, with $d\geq 2$. Let $z_1,\dots,z_{d}$ be the coordinates of $\P^{d-1}$, and $y_0,\dots, y_m$ those of $\P^m(\w)$. Choose a natural number $e$ which is a multiple of $a(\w)$. Then the closed subscheme $Z$ of $X$, defined by the maximal minors of the matrix
\[
N=\begin{pmatrixT}
    0     &    z_1     &  y_0^{e/w_0}    \\
   -z_1   &    z_2     &  y_1^{e/w_1}    \\
  \cdots  &   \cdots   & \cdots  \\
-z_{d-1}  &    z_d     & y_{d-1}^{e/{w_{d-1}}} \\
     -z_d &     0      & y_{d}^{e/w_d} 
\end{pmatrixT}
\]
is reduced. 
\end{lem}
\begin{proof}
The scheme $Z$ is the first degeneracy locus of the morphism of vector bundles
\[
\begin{tikzcd}
\O^{d+1} \rar[" N \boldsymbol{\cdot} "] & \O(1,0)^2\oplus\O(0,e).
\end{tikzcd}
\]
We observe that either $Z$ is irreducible, or has two irreducible components, one of which has necessarily dimension $m$ (and so codimension $d-1$). In the second case, the other irreducible component would have dimension $m-2$ (and so codimension $d+1$). But the codimension of $Z$ in $X$ is smaller than or equal to $d-1$. Indeed, with the notation of Lemma \ref{lem:CMvar}, the scheme $Z$ is set-theoretically the intersection of $X$ and $\Sigma$ in $V(\O^{d+1}\otimes (\O(1,0)^2\oplus\O(0,e)))$, and $X$ is embedded in $V(\O^{d+1}\otimes (\O(1,0)^2\oplus\O(0,e)))$ via the section $N$. This embedding is a regular immersion since it is a section of a smooth morphism (\cite[Lemma 31.22.8]{SP24}). Then any irreducible component of $\Sigma \cap X$ has codimension at most $d-1$ in $X$, by \cite[Lemma 43.13.12]{SP24} (see also \cite[Lemma 2.7]{Ott95}).
 This implies that $Z$ is irreducible, and hence Cohen-Macaulay, as it has the expected codimension. In particular, to show that $Z$ is reduced, we only need to find a closed point $z$ at which $Z$ is reduced. Let ${p_2}_{|Z}\colon Z \to \P(\w)$ be the restriction of the projection $p_2\colon \P^1\times \P(\w)\to \P(\w)$. Consider the closed subset $C=\{y_0=y_1=\cdots=y_d=0\}\subset \P(\w)$ and the regular locus $\P(\w)_{\mathrm{reg}}$ of $\P(\w)$. We set $U=\P(\w)\setminus \{C \cup \P(\w)_{\text{reg}} \}$ and denote by $\iota$ the open immersion $U \hookrightarrow \P(\w)$. The morphism
\begin{equation}\label{basechange}
h={p_2}_{|Z} \times \iota \colon X \times_{\P(\w)} U \to U
\end{equation}
obtained from ${p_2}_{|Z}$ via base change is proper and quasi-finite, hence finite (see\ \cite[Lemma 30.21.1]{SP24}). Since $Z \times_{\P(\w)} U \cong {{p_2}_{|Z}}^{-1}(U)$ is Cohen-Macaulay and $U$ is regular, by the miracle flatness Theorem (see\ \cite[Lemma 10.128.1]{SP24}), the morphism (\ref{basechange}) is flat. Suppose that $z \in {{p_2}_{|Z}}^{-1}(U)$ is a point at which $h$ is unramified, then $h$ is étale at $z$, and, by \cite[Corollary 3.24]{Liu06}, $Z$ is regular at $z$. In particular, $z$ is a reduced point of $Z$. Hence, for our purposes, we need to find a closed point $z$ at which $h$ is étale. It suffices to find a scheme-theoretic fibre $U_x$ via $h$ containing a reduced point. Consider the point $\overline{x}=(1,1,\dots,1) \in U$. We claim that the fibre $U_{\overline{x}}$ has a reduced point. We restrict ourselves to the chart $\{z_1\neq 0\}$. We set \[p(T)=\frac{(T-1)^{d+1}-1}{T-2}.\] In this chart, a few calculations show that the scheme-theoretic fibre of $h$ over  $\overline{x}$ is isomorphic to 
\[
\mathrm{Spec}(\K[T]/p(T))\cong \mathrm{Spec}(\K\times \cdots \times \K),
\]
which is a reduced scheme. This concludes the proof.
\end{proof}
\begin{prop}\label{prop:genmembermds}
Consider a product $X=\P^1 \times \P^m(\w)$, with $m\geq 3$. If $\mathscr{O}_X(d,e)$ is an effective line bundle with $2 \leq d \leq n$ and $e\geq a(\underline{w})$, the general member $Y$ of $|\mathscr{O}_X(d,e)|$ admits a small $\Q$-factorial modification $Y\DashedArrow Y'$.
\end{prop}
\begin{proof}
First of all, in the statement, by general we mean:
\vspace{0.1cm}

$\bullet$ $Y$ is a normal projective variety of Picard number $2$ (hence $\Q$-factorial).
\vspace{0.1cm}

$\bullet$ If $Y=\{x_0^df_0+x_0^{d-1}x_1f_1+\cdots+x_0x_1^{d-1}f_{d-1}+x_1^df_d=0\}$, the closed subscheme $\{f_i=0\}_i\subset \P^m(\underline{w})$ has codimension $d+1$ in $\P^m(\underline{w})$ and is reduced, when non-empty.
\vspace{0.1cm}

Indeed, the subset of $|\mathscr{O}_X(d,e)|$ made of the members satisfying the conditions above contains an open dense subset. To make this clear, consider the globally generated vector bundle $\mathscr{E}=\oplus_{i=0}^{d} \O_{\P(\w)}(e)$. Then, by the Bertini-type Theorem for globally generated vector bundles (see for example \cite{Ott95}), the zero locus of a general section of $H^0(\P(\w),\mathscr{E})$ has codimension $d+1$. Moreover, by the Grothendieck-Lefschetz Theorem (see\ \cite[Theorem 1]{RS06}) and the classical Bertini Theorem, the general member of $|\mathscr{O}_X(d,e)|$ has Picard number 2 and is a normal projective variety. This tells us that the general member of $|\mathscr{O}_X(d,e)|$ satisfies the two conditions above. 
\vspace{0.1cm}

We want to show that $Y$ admits a small $\Q$-factorial modification $Y\DashedArrow Y'$ (up to shrinking the open subset of $|\mathscr{O}_X(d,e)|$ describing the generality of $Y$). Let $x_0, x_1$ be the coordinates of $\P^1$, which generate $\mathscr{O}_X(1,0)$ as sections. Then we can write
\[
Y=\{x_0^df_0+x_0^{d-1}x_1f_1+\cdots+x_0x_1^{d-1}f_{d-1}+x_1^df_d=0\},
\]
where $\{f_i\}_{i=0}^d$ is a set of (weighted) homogeneous polynomials of degree $e$ in the variables $y_0,\dots,y_m$. We observe that $Y$ is the degeneracy locus of a morphism of vector bundles on $\P^1 \times \P^m(\w)$. In particular, $Y$ is the first degeneracy locus of the morphism 
\[
\O^{d+1}\to \O(1,0)^d\oplus\O(0,e),
\] 
defined by the matrix
\[
M=\begin{pmatrixT}
    x_1   &   0    & \vdots &  0  & f_0 \\
   -x_0   &  x_1   & \ddots & \hdots  & f_1 \\
     0    &  -x_0   & \ddots &   0   &   \cdots      \\
  \cdots  & \cdots & \ddots&  x_1   & f_{d-1} \\
     0    & 0     & \vdots & -x_0  & f_d
\end{pmatrixT}.
\]
Now, we observe that, unless $x_0=x_1=0$, the rank of $M$ decreases at most by 1, and so $\mathrm{ker}(M)$ is at most 1-dimensional. Let $(x_0,x_1,y_0,\dots,y_m)$ be a point of $Y$ so that $\mathrm{ker}(M)$ is one dimensional, and generated by a unique element of the form $(z_1,\dots,z_d,1)$. We observe that the point $(z_1,\dots,z_d,y_0,\dots,y_d)$ satisfies the equations defined by the maximal minors of the matrix
\[
N=\begin{pmatrixT}
    0     &    z_1     &  f_0    \\
   -z_1   &    z_2     &  f_1    \\
  \dots  &   \dots   & \dots  \\
-z_{d-1}  &    z_d     & f_{d-1} \\
     -z_d &     0      & f_{d} 
\end{pmatrixT}.
\]
Similarly, if $(z_1,\dots,z_d)\neq \underline{0}$, the rank of $N$ decreases at most by 1. Moreover, if the rank of $N$ at a point $(z_1,\dots,z_d,y_0,\dots,y_m)$ is 2, there exists a unique element in $\ker(N)$ of the form $(x_0,x_1,1)$ which satisfies the condition $\mathrm{det}(M)=0$. Let us denote by $Y'$ the closed subscheme of $X'=\P^{d-1}\times \P^m(\w)$ defined by the vanishing of the maximal minors of $N$. The scheme $Y'$ is the first degeneracy locus of the morphism of vector bundles
\[
\O^{d+1}\to \O(1,0)^2\oplus\O(0,e)
\]
defined by the matrix $N$. Then we have a well-defined rational map $f\colon Y\dashrightarrow Y'$ of schemes which, set theoretically, is defined as  
\[
f\colon (x_0,x_1,y_0,\dots,y_n) \mapsto (z_1,\dots,z_d,y_0,\dots,y_d). 
\]
In this way, we obtain a bijective correspondence between the matrices of the form of $M$, and the matrices of the form of $N$. These form in $H^0(X,E^{\vee}\otimes F)$ (resp.\ $H^0(X',{E'}^{\vee}\otimes {F'})$) an affine subspace $\A^l$, where $l=(d+1)\cdot h^0\left(\O_{\P(\underline{w})}(e)\right)$. To be more precise, we have an isomorphism of affine spaces
\begin{equation}\label{correspondence}
\begin{tikzcd}
H^0(X,E^{\vee}\otimes F)\supset  \A^l \rar[r,"\cong"]& \A^l \subset H^0(X',{E'}^{\vee}\otimes {F'}),
\end{tikzcd}
\end{equation}
defined in the obvious way.
\vspace{0.2cm}

Now, we want to show that $Y'$ is a normal variety. The indeterminacy locus of $f$ is contained in the set $\{f_0=f_1=\cdots=f_d=0\}$. Since $Y$ is general, we know that the indeterminacy locus of $f$ has codimension at least $d$ in $Y$. By the generality of the $f_i$, arguing as in Lemma \ref{lem:reduced}, it follows that $Y'$ is irreducible, and so is Cohen-Macaulay. Now, we claim that if the irreducible scheme $Y'$ is reduced, then it is a normal variety. Indeed, $f$ is bijective where it is defined and has an inverse $f^{-1}\colon Y' \dashrightarrow Y$ if $Y'$ is reduced. Moreover, if $Y'$ is reduced, the map $f$ is an isomorphism in codimension $1$, hence the scheme $Y'$ is regular in codimension 1, and so normal (since $Y'$ is Cohen-Macaulay). In particular, $f$ is a small $\Q$-factorial modification.
For convenience, we set $E=\O_X^{d+1}$, $F=\O_X(1,0)^d\oplus \O_X(0,e)$, $E'=\O_{X'}^{d+1}$,  $F'=\O_{X'}(1,0)^2\oplus \O_{X'}(0,e)$. Moreover, we denote by $V$ the incidence correspondence 
\[V=\{(x,\phi)\in (\P^{d-1}\times \P^m(\w) ) \times H^0(X,E'^{\vee}\otimes F') \;|\; \text{rk}(\phi_x)\leq 2 \}.\]


By Lemma \ref{lem:CMvar}, the scheme $V$ is a Cohen-Macaulay closed subscheme of $$(\P^{d-1}\times \P^m(\w) ) \times H^0(X,E'^{\vee}\otimes F')$$ of codimension $d-1$.

\underline{Suppose that $1<d<m$.} We show that the morphism of vector bundles associated with the matrix
\[
N'=\begin{pmatrixT}
    0     &    z_1     &  y_0^{e/w_0}    \\
   -z_1   &    z_2     &  y_1^{e/w_1}    \\
  \hdots  &   \hdots   & \hdots  \\
-z_{d-1}  &    z_d     & y_{d-1}^{e/{w_{d-1}}} \\
     -z_d &     0      & y_{d}^{e/w_d} 
\end{pmatrixT}
\]
is sufficiently general in the affine space $\A^l$, which we recall parametrises elements of $H^0(X,E'^{\vee}\otimes F')$ of the form 

\[
\begin{pmatrixT}
    0     &    z_1     &  f_0    \\
   -z_1   &    z_2     &  f_1   \\
  \hdots  &   \hdots   & \hdots  \\
-z_{d-1}  &    z_d     & f_{d-1} \\
     -z_d &     0      & f_d
\end{pmatrixT}.
\]

To this end, consider the morphism $g\colon V\to H^0(X',{E'}^{\vee}\otimes {F'})$. Since $1<d<m$, $g$ is dominant. Indeed, consider the globally generated vector bundle \[\mathscr{E}=\O_{X'}(0,e)^{d+1}.\] By the Bertini-type Theorem for globally generated vector bundles, the zero locus of a general section of $\mathscr{E}$ has the right codimension in $\P^{d-1}\times \P(\w)$. This implies that the general member $\phi$ of $H^0(X,E'^{*}\otimes F')$ has a non-empty degeneracy locus $D_2(\phi)$. In particular, $g$ is surjective, since it is projective. Let $U$ be the open subset of $H^0(X',{E'}^{\vee}\otimes {F'})$ where the fibres of $g$ have the expected dimension $m$ (see\ \cite[Corollary 13.1.5]{EGA4.3} or \cite[Theorem 9.9(b)]{Milne}). By \cite[Theorem 15.1]{Matsumura87}, the fibres of $g$ at the points of $U$ are also equidimensional. Then, since $V$ is Cohen-Macaulay, by the miracle flatness Theorem the morphism $g_{|g^{-1}(U)}\colon g^{-1}(U) \to U$ is flat. Also, $g_{|g^{-1}(U)}$ is projective since it is the base change of a projective morphism. Then by \cite[Appendix E.1, item (11)]{GW10}, the locus $U' \subset U$ where the fibre of $g_{|g^{-1}(U)}$ is reduced is open. It follows that the element $N'$ lies in $U'$, and hence the general fibre of $g_{|g^{-1}(U)}$ is reduced. Since the general member of $H^0(X,\O(d,e))\cong \A^l \subset H^0(X,E^{\vee}\otimes F)$ is a normal and $\Q$-factorial projective variety, then also the general element of $\A^l\subset  H^0(X,E'^{\vee}\otimes F')$ is, via (\ref{correspondence}).

\underline{Suppose $d=m$.} Since $Y$ is general, we may assume that $f_0,\dots,f_{m-2}$ define an integral subscheme of $\P(\w)$, and form a regular sequence. Consider the projection $p_1'\colon Y'\to \P^{d-1}$ and consider the locus $U$ of $\P^{d-1}$ where the fibres of  $p_1'$  have the expected (pure) dimension $1$. This is an open subset of $\P^{d-1}$, by (\cite[Corollary 13.1.5]{EGA4.3}). Since $Y'$ is Cohen-Macaulay, by the miracle flatness Theorem, the morphism ${p_1'}_{|{p_1'}^{-1}(U)}\colon {p_1'}^{-1}(U) \to U$ is flat. In particular, the point $q=(0\colon \dots \colon 0 \colon 1)\in \P^{d-1}$ belongs to $U$, by our choice of the $f_i$. It follows that $p_1'$ is flat at any point of the scheme-theoretic fibre $U_q=\P^{d-1}\times_{\P(\w)} \mathrm{Spec}(k(q))$. But by our choice of $f_0,\dots,f_{m-2}$, the closed subscheme $U_q$ is reduced. This implies that $Y'$ contains a reduced point (see\ \cite[Theorem 23.9]{Matsumura87}), and hence is reduced everywhere since $Y'$ has no embedded points. We conclude that the general member of $\A^l$ is normal as we did above.
\end{proof}



\begin{rmk}
We point out that, in the proof of Proposition \ref{prop:genmembermds}, the strategy to prove the case $d=m$ works also for the case $1<d<m$, but we thought it was instructive to present both strategies. The first strategy shows that a general element of $\A^l$ is a projective variety, starting from a special element in $\A^l$ that is a projective variety. The second strategy directly shows that the general element of $\A^l$ is a projective variety.
\end{rmk}

\begin{cor}\label{cor1}
Let $X$ be a product $\P^1\times \P^m(\w)$, and $\O_X(d,e)$ an effective line bundle with $1\leq d\leq n$ and $e\geq a(\w)$. Then either $Y$ admits a unique small $\Q$-factorial modification, which is an isomorphism if $d=m$, or a unique divisorial contraction. Moreover, if $1<d<m$, we have
\[
\mathrm{Eff}(Y)=\mathrm{Mov}(Y)=\R_{\geq 0}H_1+\R_{\geq 0}(eH_2-H_1).
\] 
If $d=m$, we have
\[
\mathrm{Eff}(Y)=\mathrm{Nef}(Y)=\R_{\geq 0}H_1+\R_{\geq 0}(eH_2-H_1).
\] 
Id $d=1$, we have $$\mathrm{Eff}(Y)=\R_{\geq 0}H_1+\R_{\geq 0}(eH_2-H_1)$$ and $$\mathrm{Mov}(Y)=\mathrm{Nef}(Y)=\R_{\geq 0}H_1+\R_{\geq 0}H_2.$$
In any of these cases, $Y$ is a Mori dream space.
\end{cor}
\begin{proof}
Consider the small $\Q$-factorial modification $f\colon Y \dashrightarrow Y'$ constructed in Proposition \ref{prop:genmembermds}. Given the two projections $p_1'\colon Y' \to \P^{d-1}$ and $p_2'\colon Y' \to \P(\w)$, we denote by $H_1'$ (resp.\ $H_2'$) the pull-back ${p_1'}^*(\mathscr{O}_{\P^{d-1}}(1))$ (resp.\ the reflexive sheaf associated with the push-forward of $H_2$ via $f$).
 We compute $f_*(H_1)$ and $f_*(eH_2)$. The interesting case is $f_*(H_1)$. To compute it, pick for example $\{x_0=0\} \cap Y$. The image of this divisor of $Y$ via $f$ is contained in the set of points of $Y'$ satisfying the maximal minors of the matrix
\[
\begin{pmatrixT}
    0     &    z_1     &  f_0    \\
   -z_1   &    z_2     &  f_1    \\
  \hdots  &   \hdots   & \hdots  \\
-z_{d-1}  &    z_d     & f_{d-1} \\
     -z_d &     0      & 0
\end{pmatrixT}.
\]
Since we started from points of $Y$ having $x_0=0$, it follows that the image of $\{x_0=0\}\cap Y$ via $f$ is the zero locus of a section of $eH_2'-H_1'$, and so $f_*(H_1)=eH_2'-H_1'$. It is straightforward to check that $f_*(eH_2)=eH_2'$. Note that $eH_2-H_1$ is not big, for example, because $f_*(eH_2-H_1)=H'_1$ is not big in $Y'$. Then the ray $\R_{\geq 0}(eH_2-H_1)$ is extremal in $\mathrm{Eff}(Y)$. Now, the Picard number of $Y$ is $2$, and so $\mathrm{Eff}(Y) =\R_{\geq 0}(eH_2-H_1)+\R_{\geq 0}H_1$. Any nef divisor on $Y$ is semiample. Indeed, if $d<n$, we have $\mathrm{Nef}(Y)=\R_{\geq 0} H_2+ \R_{\geq 0} H_1$, and $H_2,H_1$ are semiample. If $d=m$, since the map $f \colon Y \DashedArrow Y'$ is an isomorphism, the Cartier divisor $eH_2-H_1$ is nef, and induces a morphism $Y \to \P^{d-1}$. This is enough to conclude that $Y$ is a Mori dream space. So far we have dealt with the case $2\leq d \leq n$. Assume now $d=1$. In this case, $Y$ is the blow-up of $\P(\w)$ along the ideal sheaf $\mathscr{I}$ generated by $f_0,f_1$, and we have a divisorial contraction $p\colon Y \to \P(\w)$. The exceptional divisor is given by $eH_2-H_1$. Indeed, an equation for $Y$ is $x_0f_0+x_1f_1=0$. Then, if $f_0=0$, either $x_1=0$, or $f_1=0$. Taking out $\{x_1=0\}$ (which means taking a section of $eH_2-H_1$), gives exactly $f_0=f_1=0$, i.e.\ the exceptional divisor. Also in this case $Y$ is a Mori dream space because all the conditions in the definition of Mori dream space are satisfied. Moreover, $\mathrm{Mov}(Y)=\mathrm{Nef}(Y)=\R_{\geq 0} H_2+ \R_{\geq 0} H_1$.
\end{proof}

\begin{prop}\label{prop:coxring}
Let $X$ be a product $\P^1\times \P^m(\w)$, with $m\geq 3$. Consider a line bundle $\O(d,e)$, with $1\leq d\leq m$, $e= k\cdot a(\w)$, and $k\geq2$. Then the Cox ring of the general member 
$$Y=\{x_0^df_0+x_0^{d-1}x_1f_1+\dots+x_0x_1^{d-1}f_{d-1}+x_1^d=0\}$$
of $|\O(d,e)|$ is 
$$\mathrm{Cox}(Y)\cong \K[x_0,x_1,y_0,\dots,y_n,z_1,\dots,z_d]/I,$$
where $I=(f_0+x_0z_1,f_1-x_0z_1+x_1z_2,\dots,f_{d-1}-x_0z_{d-1}+x_1z_d,f_d-x_0z_d)$.
\end{prop}
\begin{proof}
Set $X'=\P^1\times \P^n$ and consider the projection $\pi \colon X' \to X$. The pull-back $\pi^*|\O_X(d,e)|$ of the complete linear series $|\O_X(d,e)|$ is an ample base point-free linear series $V\subset |\O_{X'}(d,e)|$. A straightforward computation shows that $H^0(Y,eH_2-H_1)=d$. Let $z_1,\dots,z_d$ be a basis for this space of global sections. 
By the Bertini Theorem and the Grothendieck-Lefschetz Theorem (\cite[Theorem 1]{RS06}), the pull-back $Y'=\pi^*(Y)$ of the general member $Y$ of $|\O_X(d,e)|$ is a smooth projective variety of Picard number 2. The Cartier divisor $Y'=\pi^*(Y)$ coincides with the inverse image scheme $X'\times_X Y$ under $\pi$. Now, since by Corollary \ref{cor1} $Y$ is a Mori dream space, by \cite[Theorem 3.1]{Okawa} also $Y'$ is a Mori dream space. Set $H_1'=\O_{X'}(1,0)_{|Y'}$ and $H_2'=\O_{X'}(0,1)_{|Y'}$. Moreover, let $x_0',x_1'$ and $\{y_j'\}_j$ be the coordinates of $\P^1$ and $\P^m$ respectively. Note that $H^0(Y',eH_2'-H_1')=d$. By \cite[Theorem 1.1(iv)]{Ottem15}, the Cox ring of $Y'$ is $$\mathrm{Cox}(Y')\cong \K[x_0',x_1',y_0',\dots,y_n',z_1',\dots,z_d']/I',$$
where $$I'=(f_0'+x_0'z_1',f_1'-x_0'z_1'+x_1'z_2',\dots,f_{d-1}'-x_0'z_{d-1}'+x_1'z_d',f_d'-x_0'z_d'),$$
and $z_1',\dots,z_d'$ have degree $eH_2'-H_1'$. Now, consider the vector bundle $\mathscr{E}=\O_X(1,0)^d\oplus \O_X(0,e)$ and the associated projective bundle $q\colon \P(\mathscr{E})\to \P^1 \times \P(\underline{w})$. Recall that the ideal of $Y$ in the toric variety $\P(\mathscr{E})$ is $$I_Y=(f_0+x_0z_1,f_1-x_0z_1+x_1z_2,\dots,f_{d-1}-x_0z_{d-1}+x_1z_d,f_d-x_0z_d,t),$$ where $t$ is (up to scalars) the unique section of $H^0(\P(\mathscr{E}),\O_{\P(\mathscr{E})}(1)\otimes q^*(\O(0,-e)))$, whose zero locus is the exceptional divisor of the unique extremal divisorial contraction of $\P(\mathscr{E})$. 
It follows that the ring $$R=\mathrm{Cox}(\P(\mathscr{E}))/(I_Y)\cong \K[x_0,x_1,y_0,\dots,y_d,z_1,\dots,z_d]/I$$ is an integral domain. Now, choose for $\mathrm{Cox}(Y')$ the grading induced by $\{H_1,H_2\}\subset \mathrm{Cl}(Y)$.
Firstly, we observe that we have a well-defined ring homomorphism $R\to \mathrm{Cox}(Y')$, defined by $x_i\mapsto x_i'$, $z_i\mapsto z_i'$ and $y_i\mapsto (y_i')^{\frac{a(\underline{w})}{w_i}}$. Secondly, we note that the ring homomorphism $R\to \mathrm{Cox}(Y')$ factors through $\mathrm{Cox}(Y)$, so that we have the factorisation $R\to \mathrm{Cox}(Y) \to \mathrm{Cox}(Y')$, where $\mathrm{Cox}(Y) \to \mathrm{Cox}(Y')$ is obtained by pull back of Weil divisors under $Y'\to Y$. The latter is a graded morphism of graded rings, with respect to the gradings we have chosen for $\mathrm{Cox}(Y)$ and $\mathrm{Cox}(Y')$. The image of $\mathrm{Cox}(Y) \to \mathrm{Cox}(Y')$ is the subring $$\K\left[x_0',x_1',(y_0')^{w_0},\dots,(y_n')^{w_n},z_1',\dots,z_d'\right]/I'.$$ The Cox ring of $Y$ has Krull dimension equal to $\mathrm{dim}(Y)+\rho(X)$. Then $R\to \mathrm{Cox}(Y)$ is a surjective morphism of integral domains with the same Krull dimension $m+2$.  This implies $R\cong \mathrm{Cox}(Y)$.
\end{proof}

\begin{prop}\label{prop:notMDS}
Let $X$ be a product $\P^1\times \P^m(\w)$, with $m\geq 3$, and $\O(d,e)$ a line bundle, with $d>m$, and $e\geq 2$. Then the very general member of $|\O_X(d,e)|$ is not a Mori dream space.
\end{prop}
\begin{proof}
Set $X'=\P^1\times \P^m$. Consider the morphism $g \colon X\to X'$, defined by
\[
g\colon (x_0,x_1,y_0,\dots,y_m) \mapsto (x_0,x_1,y_0^{a(\w)/w_0},\dots,y_m^{a(\w)/w_m}).
\]
Suppose that $d>m$ and $e\geq 2\cdot a(\w)$. Recall that $a(\w):=\mathrm{l.c.m.}\{w_i\}_i$.
We observe that the complete and base point-free linear series $|\O_{X'}(d,e/a(\w))|$ is pulled back to a base point-free linear series $V\subset |\O_X(d,e)|$. By \cite[Theorem 5.6]{Ottem15}, the very general member $Y'$ of $|\O_{X'}(d,e/a(\w))|$ is a normal variety and is not a Mori dream space, because its effective cone is not closed. To be more precise, the divisor $\frac{me}{a(\w)}H_2- dH_1$ on $Y'$ is pseudo-effective (actually nef), but it has negative Itaka dimension. By the very generality, the hypersurface $Y'$ is pulled back via $g$ to an irreducible and normal member of $V$, by \cite[Theorem 1]{RS06} and the Bertini Theorem. Recall that $Y=g^*(Y')$ coincides with the inverse image scheme $X\times_{X'} Y'$. The Cartier divisor  $g^*(\frac{me}{a(\w)}H_2'- dH_1')=meH_2-dH_1$ on $Y$ is pseudo-effective. But, by Proposition \ref{prop:Itakadim}, we have $\kappa(meH_2-H_1)=-\infty$. Then $Y$ is not a Mori-dream space. It follows from the semicontinuity Theorem (see\ \cite[Theorem III.12.8]{Har}) and \cite[Proposition 1.4.14]{Laz} that the very general member of $|\O(d,e)|$ is not a Mori dream space. To be more precise, the very general deformation of $Y$ in its linear series has a non-closed effective cone. 
\end{proof}

We conclude this section by proving Theorems \ref{thmA} and \ref{thmB}.

\begin{proofA}
Item (1) is the content of Proposition \ref{prop1}, item (2) is the content of 
Corollary \ref{cor1}, and Propositions \ref{prop:coxring}, and item (3) is the content of \ref{prop:notMDS}. 
To show item (4), set $X'=\P^1\times \P^n$ and consider the morphism $g \colon X \to X'$, defined by
\[
g\colon (x_0,x_1,y_0,\dots,y_m) \mapsto (x_0,x_1,y_0^{a(\w)/w_0},\dots,y_m^{a(\w)/w_m}),
\]
as in Proposition \ref{prop:notMDS}. Let $Y'$ be a general hypersurface in $|\O_{X'}(d,1)|$. We know that $Y'$ is a projective bundle over $\P^1$, as the Brauer group of $\P^1$ is trivial. In particular, $Y'$ is a toric variety (and so a Mori dream space), as every vector bundle on $\P^1$ splits as a direct sum of line bundles. We can assume that $Y'$ is pulled back to a normal and $\Q$-factorial member in the linear series $V=g^*|\O_{X'}(d,1)|\subset |\O_X(d,a(\underline{w}))|$, with class group of rank $2$. Since $Y'$ is a Mori dream space and the class group of $Y$ has rank 2, also $Y$ is a Mori dream space, by \cite[Theorem 3.1]{Okawa}. This concludes the proof of the theorem.
\end{proofA}

The main difference between Theorem \ref{thmA} and Theorem \ref{thmB} is that for threefolds we do not have the Grothendieck-Lefschetz Theorem \cite[Theorem 1]{RS06}, so a priori we cannot say anything about the class group of a general member of an ample and base point-free linear series. However, we have a weaker result, the Noether-Lefschetz Theorem \cite[Theorem 1]{RS09}, which, under certain assumptions, provides information about the class group of a very general member of an ample and base point-free linear series on a normal threefold.

\begin{proofB}
Under the assumption that $K_X$ is Cartier, so that $X=\P(1,1,2)$ or $X=\P(1,2,3)$, the hypotheses of \cite[Theorem 1]{RS09} are satisfied in the first case if and only if $d\geq 2$ and $e\geq 4$, and in the second case if and only if $d\geq 2$ and $e\geq 6$. 
\begin{itemize}
\item Suppose $X=\P(1,1,2)$. If $d=2$ and $e=4k$, we argue as in Proposition \ref{prop:genmembermds} and Corollary \ref{cor1} to conclude that the very general member of $|\O_X(2,2k)|$ is a Mori dream space and to compute the various cones. We argue as in Proposition \ref{prop:coxring} to compute its Cox ring. If $d\geq3$ and $e=2k$, we argue as in Proposition \ref{prop:notMDS} to conclude that the very general member of $|\O_X(d,2k)|$ is not a Mori dream space. If $d=1$ and $e=2$, or $d=2$ and $e=2$, we consider the quotient morphism $\P^1\times \P^2 \to \P^1\times \P(1,1,2)$ and observe that the general member of $|\O_X(1,2)|$ (resp.\ $|\O_X(2,2)|$) is pulled back to a Del Pezzo surface of degree 5 (resp.\ 2). Then we use \cite[Theorem 1.1]{Okawa} to conclude that the general member of $|\O_X(1,2)|$ (resp.\ $|\O_X(2,2)|$) is a Mori dream space.
\item Suppose $X=\P(1,2,3)$. If $d=2$ and $e=6k$, we argue as in Proposition \ref{prop:genmembermds} and Corollary \ref{cor1} to conclude that the very general member of $|\O_X(2,6k)|$ is a Mori dream space and to compute the various cones. We argue as in Proposition \ref{prop:coxring} to compute its Cox ring. If $d\geq3$ and $e=6k$, we argue as in Proposition \ref{prop:notMDS} to conclude that the very general member of $|\O_X(d,e)|$ is not a Mori dream space.
\end{itemize}
\end{proofB}

\subsection{The birational geometry of $Y$ via toric geometry}\label{toricvariety}

The purpose of this subsection is to describe the toric variety inducing the birational geometry of a general hypersurface $Y$ of a certain bidegree in $\P^1\times \P(\underline{w})$.
\medskip

Let $Y$ be a general hypersurface in a linear series $|\O_X(d,e)|$, where $X=\P^1 \times \P^m(\underline{w})$ is singular, $m\geq 3$, $1\leq d<n$, and $e=ka(\underline{w})$ for some positive integer $k$. Consider the vector bundle $\mathscr{E}=\O_X(1,0)^d\oplus\O_X(0,e)$. Set $h^0(\P(\underline{w}),\O_{\P(\underline{w})}(e))=N+1$. The vector bundle $\mathscr{E}$ is globally generated, and as generators we can choose $(0,\dots,0,s_0),\dots,(0,\dots,0,s_N)$, where $\{s_i\}_i$ is a basis for $H^0(X,\O_X(0,e))$, and, if $\O_X(1,0)$ occupies the $i$th position in $\mathscr{E}$, so that $1\leq i\leq d$, we add $(0,\dots,x_0,\dots,0)$ and $(0,\dots,x_1,\dots,0)$ to the set of generators. Here $x_0,x_1$ are the coordinates of $\P^1$.
\medskip

Then we have a surjective morphism of vector bundles $\O_X^{2d+N+1}\to \mathscr{E}$, which in turn induces a closed immersion $\P(\mathscr{E})\hookrightarrow \P^1\times \P(\underline{w}) \times \P^{2d+N}$. The equations of $\P(\mathscr{E})$ in $\P^1\times \P(\underline{w}) \times \P^{2d+N}$ are the minors of the matrix 
\[
\begin{pmatrix}
    x_0     &    u_{01}  & \cdots  &  u_{0d} \\   
    x_1     &    u_{11} & \cdots  &  u_{1d}  

\end{pmatrix},
\]
where $u_{ji}$, $j=0,1$, $i=1,\dots,d$ are the $2d$ coordinates of $\P^{2d+N}$, together with $I_e$, and the minors of the matrix

\[
\begin{pmatrix}
    s_0     &    s_1 & \cdots   &  s_N \\   
    u_0     &    u_1    & \cdots   &   u_N  

\end{pmatrix},
\]
where $u_0,\dots,u_N$ are the coordinates of $\P^N$.

\begin{prop}\label{Zisnormal}
Following the notation of the discussion above, and keeping the assumptions in there, suppose that $\O_{\P(\underline{w})}(e)$ is very ample and that $\nu_e(\P(\underline{w})) \subset \P^N$ is arithmetically Cohen-Macaulay. Then the natural map $c\colon \P(\mathscr{E}) \to \P^1\times \P^{2d+N}$ (obtained by composition) is an extremal divisorial contraction onto its scheme-theoretic image $Z$. 
\end{prop}
\begin{proof}
First of all, we observe that the subvariety $W$ of $\P^{2d+N}$ defined by 
\[
\begin{pmatrix}
    u_{01}  & \cdots  &  u_{0d} \\   
    u_{11} & \cdots  &  u_{1d}  

\end{pmatrix}
\]
and $I_e$ is the join of the Segre variety $\Sigma_{1,d-1}$, i.e. the image of the Segre embedding $$\P^1\times \P^{d-1}\to \P^{2d-1},$$ and $\nu_e(\P(\underline{w}))\subset \P^N$. In particular, $W$ is an arithmetically Cohen-Macaulay projective variety, because $\nu_e(\P(\underline{w}))$ and $\Sigma_{1,d-1}$ are. Hence, the subvariety $Z$ of $\P^1 \times \P^{2d+N}$ defined by 
\[
\begin{pmatrix}
    x_0     &    u_{01}  & \cdots  &  u_{0d} \\   
    x_1     &    u_{11} & \cdots  &  u_{1d}  

\end{pmatrix}
\]
and $I_e$ is the blowing-up of $W$ along the ideal sheaf generated by $u_{01},u_{11}$. In particular, $Z$ is an integral scheme and hence coincides with the scheme-theoretic image of $c$. We observe that the fibres of $c$ over the points of $Z$ are connected, as they are either a single point or isomorphic to $\P(\underline{w})$. The locus where the dimension of the fibres of $c$ jumps is $V=\{u_0=\cdots=u_N=0\}$. It follows that the restriction of $c$ to $\P(\mathscr{E})\setminus c^{-1}(V)$ is an isomorphism onto $Z \setminus V$, since there it has an inverse. Consequently, $Z$ is regular in codimension 1, because $\P(\mathscr{E})$ is normal. Note that the exceptional locus of $c\colon \P(\mathscr{E}) \to Z$ is isomorphic to $\P(\underline{w})\times \Sigma_{1,d-1}$
Then, to conclude that $c$ is an extremal divisorial contraction, we have to show that $Z$ is normal at the points contained in $c^{-1}(V)$. To this end, we consider the morphism $Z\to \P^1$. This is a flat isotrivial fibration, whose fibres are isomorphic to a cone over $\nu_e(\P(\underline{w}))$. Since $\nu_e(\P(\underline{w}))$ is arithmetically Cohen-Macaulay by our assumption, all the fibres of $Z\to \P^1$ are Cohen-Macaulay. In particular, we have a flat fibration over a Cohen-Macaulay base with Cohen-Macaulay fibres, from which we infer that $Z$ is Cohen-Macaulay. Consequently, $Z$ satisfies Serre's $R_1$ and $S_2$ conditions, hence it is normal.
\end{proof}

\begin{prop}\label{flip}
Let $Z$ and $W$ be the varieties in Proposition \ref{Zisnormal}. Let $f^+\colon Z^+ \to W$ be the blowing-up of $W$ along the ideal sheaf generated by $u_{01},\dots,u_{0d}$. Then $Z^+$  is a normal and $\Q$-factorial projective variety, and the induced diagram  
\begin{center}
\begin{tikzcd} 
Z\arrow[dr,"f"] \arrow[rr, dashed, "\psi"] &   & Z^+ \arrow[dl, "f^+"] \\
                       & W &
\end{tikzcd}
\end{center}
is a flip diagram, if $d>1$. Instead, if $d=1$, then $f$ is a divisorial contraction.
\end{prop}
\begin{proof}
We start by observing that $W$ is normal, as it is arithmetically Cohen-Macaulay by the previous proposition, and regular in codimension 1, because it is birational to $Z$. Suppose that $d>1$. We note that $Z^+$ is a closed subvariety of $\P^{d-1}\times \P^{2d+N}$, and its equations in there are the $2\times 2$ minors of the matrix
\[
\begin{pmatrix}
    z_1     &     \cdots & z_d    \\
    u_{01}   &    \cdots &  u_{0d}    \\
    u_{11}   &    \cdots &  u_{1d}
\end{pmatrix},
\]
 along with the generators of $I_e$. We claim that the morphisms $f$ and $f^+$ are small birational contractions. Indeed, the exceptional locus of both is the preimage of $\{u_{ji}=0\}_{j,i}$. To be more precise, the exceptional locus of $f$ is isomorphic to $\P^1\times \P(\underline{w})$, which has codimension $d$ in $Z$, and that of $f^+$ is isomorphic to $\P^{d-1}\times \P(\underline{w})$, and so it has codimension $2$ in $Z^+$. To show that $Z^+$ is normal, it is then sufficient to show that it is normal along the divisor $E=\{u_{0i}=0\}_{i}\subset Z^+ $. Note that this is a Cartier divisor since it coincides with $\mathscr{I}\mathscr{O}_{Z^+}=\mathscr{O}_{Z^+}(1)$, where $\mathscr{I}$ denotes the ideal sheaf generated by $\{u_{0i}\}_{i}$. We now claim that $E$ is a Cohen-Macaulay variety. Let $W'$ be the subvariety of $W$ defined by the ideals $\{u_{0i}\}_i$ and $I_e$. This is a cone over  $\nu_e(\P(\underline{w}))$. The morphism $E\to W'$ is the blowing-up of $W'$ along the ideal sheaf generated by $\{u_{1i}\}_{i=1}^n$. By our assumption on $\nu_e(\P(\underline{w}))$, $E$ is Cohen-Macaulay away from the exceptional locus of $E\to W'$. The exceptional locus of this morphism is a Cartier divisor isomorphic to $\P^{d-1}\times \P(\underline{w})$, and so a Cohen-Macaulay variety. It follows by \cite[Theorem 2.3]{BH98} that $E$ is Cohen-Macaulay. Always by \cite[Theorem 2.3]{BH98}, the variety $Z^+$ is Cohen-Macaulay at the points of $E$, and hence is Cohen-Macaulay everywhere. Now, $Z^+$ is regular in codimension 1, since it is isomorphic in codimension 1 to $Z$, and from this, we infer that $Z^+$ is a normal projective variety. Furthermore, since $Z$ and $Z^+$ have isomorphic class groups (because they are isomorphic in codimension 1), it follows that $Z^+$ has Picard number $2$, and so is $\Q$-factorial. In particular, $\psi\colon Z \dashrightarrow Z^+$ is a small $\Q$-factorial modification. To conclude that $\psi$ is a flip, it is sufficient to find a Cartier divisor on $Z$ which is $f$-ample and whose strict transform via $\psi$ is $-f^+$-ample. Let $D_1$ be the restriction of $\O_{\P^{1}\times \P^{2d+N}}(1,0)$ to $Z$. Its strict transform via $\psi$ is the Cartier divisor $E$, defined above, which then satisfies $h^0(W^+,\O_{W^+}(E))=2$. As in the case of $Z$, it is easy to check that $Z^+$ is the codomain of an extremal divisorial contraction $\P(\mathscr{E}^+)\to Z^+$, where $$\mathscr{E^+}=\O_{\P^{d-1}\times \P(\underline{w})}(1,0)^2\oplus \O_{\P^{d-1}\times \P(\underline{w})}(0,e)$$ and $$\P(\mathscr{E}^+)\hookrightarrow \P^{d-1}\times \P(\underline{w})\times \P^{2d+N}.$$ Now, let $D_2^+$ (resp.\ $D_1^+$) be the restriction of $\O_{\P^{d-1}\times \P^{2d+N}}(0,1)$ (resp.\ $\O_{\P^{d-1}\times \P^{2d+N}}(1,0)$) to $Z^+$. An easy calculation shows that $$H^0(Z^+,D_2^+-D_1^+)\cong H^0(\P^{d-1}\times \P(\underline{w}),\mathscr{E}\otimes \O_{\P^{d-1}\times \P(\underline{w})}(-1,0))\cong \K^2,$$ and one easily checks that indeed we have $E=D_2^+-D_1^+$. The divisor $D_1$ is $f$-ample, and $E$ is $-f^+$-ample. In particular, $\psi$ is the $K_Z+(-K_Z-D_1)$-flip of $f$, and this concludes the proof. If $d=1$, we have $W=Z^+$ and $f$ is a divisorial contraction.
\end{proof}

\begin{ex}
\begin{itemize}
\item All Veronese varieties $\nu_e(\P^m)$ are arithmetically Cohen-Macaulay, and all effective (non-linearly trivial) divisors are very ample on $\P^m$. Hence, Proposition \ref{Zisnormal} applies in the smooth case. More generally, $\nu_e(\P(1^{m+1},e))$ is arithmetically Cohen-Macaulay, and we recall that $\P(1^{m+1},e)$ is isomorphic to the (projective) cone over the $e$th Veronese variety $\nu_e(\P^m)$.

\item By \cite[Proposition 5.13]{KM98}, quotient singularities are rational. Furthermore, by \cite[p. 39, Theorem]{Dolg82}, we have $H^i(\P^m(\underline{w}),\O_{\P(\underline{w})}(l))=0$ for any $l\geq 0$ and $0<i<m$. Then, by \cite[Lemma 7.2.7]{Fuj17}, as long as $\O_{\P^m(\underline{w})}(e)$ is very ample, and $\nu_e(\P^m(\underline{w}))\subset \P^N$ is projectively normal, any Cartier hypersurface of bidegree $(d,e)$ with $1<d<m$ fulfills the hypotheses of Proposition \ref{Zisnormal}. It is worth noting that embedded 3-dimensional weighted projective spaces are always projectively normal, as shown in \cite[Theorem 1]{Ogata05}.

\end{itemize}
\end{ex}

\begin{cor}
Let $Y$ be a general Cartier hypersurface of bidegree $(d,e)$ in a product $\P^1 \times \P^m(\underline{w})$, with $1\leq d<m$, $m\geq 3$. Suppose that $\O_{\P(\underline{w})}(e)$ is very ample and that $\nu_e(\P(\underline{w}))$ is arithmetically Cohen-Macaulay. Moreover, let $\P(\mathscr{E})$, $Z$ and $W$ be the varieties described in Propositions \ref{Zisnormal}, \ref{flip}. Then $\P(\mathscr{E})$, $Z$, $Z^+$ and $W$ are toric varieties, and the birational geometry of $Y$ is induced by that of $Z$.
\end{cor}
\begin{proof}
Recall that $\mathscr{E}=\O(1,0)^d\oplus \O(0,e)$. Since this vector bundle is a direct sum of line bundles, the associated projective bundle is a toric variety. Let $p\colon \P(\mathscr{E})\to X$ be the structural morphism. It is easy to check that $$\mathrm{Cox}(\P(\mathscr{E}))\cong \K[x_0,x_1,y_0,\dots,y_m,z_1,\dots,z_d,t],$$ where the grading is as follows:
$x_i$ has degree $p^*(\O(1,0))$, $y_i$ has degree $p^*(\O(0,w_i))$, $z_i$ has degree $\O_{\P(\mathscr{E})}(1) \otimes p^*(\O(-1,0))$, and $t$ has degree $\O_{\P(\mathscr{E})}(1)\otimes p^*(\O(0,-e))$. Note that the zero locus of the section $t$ equals the exceptional divisor of the divisorial contraction $c\colon \P(\mathscr{E}) \to Z$, described in Proposition \ref{Zisnormal}. Then, by Propositions \ref{Zisnormal}, \ref{flip}, and for example \cite[Toric Contraction Theorem II]{Wisniewski02}, the varieties $Z$, $W$ are toric. When $d>1$, Theorem \cite[Toric Contraction Theorem II]{Wisniewski02} implies that also $Z^+$ is a toric variety since $Z^+$ is obtained via an extremal divisorial contraction $\P(\mathscr{E}^+)\to Z^+$, where $\mathscr{E}^+$ is as in Proposition \ref{flip}. The hypersurface $Y$ can be embedded in $Z$ via a closed immersion as follows:
the morphism of vector bundles $h\colon \O_X^{d+1}\to \O_X(1,0)^{d+1}\oplus \O_X(0,e)$, described in Proposition \ref{prop:genmembermds}, decreases of rank exactly at the points of $Y$, and exactly of 1. In particular, $(\O_Y(1,0)^{d+1}\oplus\O_Y(0,e))/\mathrm{im}(h_{|Y})$ is a line bundle, and so we have a section $Y \hookrightarrow \P(\mathscr{E}_{|Y})$. But we also have $\P(\mathscr{E}_{|Y})\hookrightarrow \P(\mathscr{E})$, so that we have a closed immersion $\iota\colon  Y\hookrightarrow \P(\mathscr{E})$.
From the defining equation of $Y$ in $\P^1 \times \P(\underline{w})$ it is easy to see that the ideal of $Y$ in $\P(\mathscr{E})$ is $$I_Y=(f_0+x_0z_1,f_1-x_0z_1+x_1z_2,\dots,f_{d-1}-x_0z_{d-1}+x_1z_d,f_d-x_0z_d,t).$$ We now observe that $Y\cap \mathrm{Exc}(c)=\emptyset$, and so $c\circ \iota\colon Y \to Z$ is a closed immersion. In particular, $Y$ is a complete intersection in $\P(\mathscr{E})$ and $Y$. Suppose that $d>1$ and consider the birational map $\psi\colon Z\dashrightarrow Z^+$, described in Proposition \ref{flip}. From the defining equations of $Y$ in $Z$, it is clear that $\mathrm{ind}(\psi)\cap Y=\{f_0=\cdots=f_d=0\}$, where $\mathrm{ind}(\psi)$ is the indeterminacy locus of $\psi$. Moreover, the image of a point $(x_0,x_1,y_0,\dots,y_m)$ via the birational map $Y\dashrightarrow Y^+$, described in Proposition \ref{prop:genmembermds}, is the tuple $(z_1,\dots,z_d,y_0,\dots,y_m)$ of $\P^{d-1}\times \P(\underline{w})$, such that $(z_1,\dots,z_d,1)$ is in the kernel of the matrix $M$ in Proposition \ref{prop:genmembermds}. But this tuple is exactly the one satisfying the generators of $I_Y$. This makes it clear that the restriction of $\psi$ to $Y$ coincides with the birational map $Y\dashrightarrow Y^+$ in Proposition \ref{prop:genmembermds}. If $d=1$, it is clear that the divisorial contraction $Z\to W$ restricted to $Y$ gives back the divisorial contraction $Y\to \P(\underline{w})$.
\end{proof}

\section{Calabi-Yau multi-weighted hypersurfaces and the cone conjecture}\label{morespaces}
We prove below Proposition \ref{propCY} and Theorem \ref{thmC}.
\begin{proofD}
Since $X$ is Gorenstein, $-K_X$ is Cartier. By the adjunction formula, the general member $Y$ of $|-K_X |$ is Calabi-Yau. First of all, we note that the Picard number of $Y$ equals that of $X$, by the Grothendieck-Lefschetz Theorem (see\ \cite[Theorem 1]{RS06}). We have $\overline{\mathrm{NE}}(Y)\subseteq \overline{\mathrm{NE}}(X)$, via proper push-forward of $1$-cycles, hence $\Nef(X)_{|Y}\subseteq \mathrm{Nef}(Y)$. Let $H_1,\dots,H_{n+k}$ be the nef $\Q$-Cartier divisors $\O(0,\dots,1,\dots,0)_{|Y}$, where the $1$ is in position $j$, for any $j=1,\dots,n+k$. Consider a linear combination (with real coefficients) $D=\sum_{j=1}^{n+k}a_jH_j$. We claim that $D$ is nef if and only if all the coefficients $a_i$ are nonnegative. If all the coefficients $a_i$ are nonnegative, the divisor $D$ is nef, as it is a sum of nef $\Q$-Cartier divisors. Suppose then that $D$ is nef. If there are no projective lines among the factors of $X$, we are done, hence we assume that there is at least one projective line $\P^1$ among the factors of $X$, so that $n\geq1$. Let $x_0,x_1$ be the coordinates of this projective line and assume that this projective line occupies the first position in $X$. Then $H_1=\O(1,0,\dots,0)_{|Y}$ is the pull-back of $\O_{\P^1}(1)$ via the projection $Y \to \P^1$. We have
\[
Y=\{x_0^2f_0+x_0x_1f_1+x_1^2f_2=0\}.
\]
As $Y$ is general and the dimension of $X$ is at least 3, the set of points of $(\P^1)^{n-1}\times \P(\w^1)\times \P(\w^2)\times \dots \times \P(\w^k)$ such that $f_0=f_1=f_2=0$ is not empty, hence there is at least one curve (actually, one smooth rational curve) contracted by the morphism \[Y\to (\P^1)^{n-1} \times \left(\prod_{i=1}^{k} \P(\w^i)\right).\] This implies that $a_1 \geq 0$. Then $\Nef(Y)\subseteq \mathrm{Nef}(X)_{|Y}$, and so $\mathrm{Nef}(Y)=\mathrm{Nef}(X)_{|Y}$.

\underline{Proof of (1)}. If we consider the short exact sequence $$0\to \O_X(-Y)\to \O_X \to \O_Y\to 0,$$ using the vanishing of cohomology in degree 1 on weighted projective spaces and the Künneth formula, one quickly realises that a line bundle $\O_Y(D)$ has global sections if and only if $D=\sum_ia_iH_i$ and any (meaningful) $a_i$ is nonnegative. Then $\mathrm{Eff}(Y)=\mathrm{Nef}(X)_{|Y}=\mathrm{Nef}(Y)$, and so we also have $\mathrm{Eff}(Y)=\mathrm{Mov}(Y)$. Moreover, any nef line bundle on $Y$ is semiample, and this is enough to conclude that $Y$ is a Mori dream space.

\underline{Proof of (2)}.
To see that $\mathrm{Nef}(Y)\subsetneq \mathrm{Eff}(Y)$ it suffices to observe that \[H^0\left(Y,\left[\sum_{i\geq 2}\left(\sum_jw_j^{i-1}\right)H_i\right]-H_1\right)\] is 2-dimensional.
\vspace{0.1cm}


Now, we prove that $Y$ admits finitely many small $\Q$-factorial modifications. Let $Y \dashrightarrow Y'$ be any such modification. Since $Y$ and $Y'$ are isomorphic in codimension 1, the canonical divisor of $Y'$ is trivial. Since $Y$ has klt singularities, also $Y'$ does. Hence, by \cite[Corollary 3.3]{Birkar11}, $Y \dashrightarrow Y'$ is a composition of flops. We now observe that $Y$ has only one flopping contraction, that is the projection $$\pi_1 \colon Y \to \prod_{i=1}^k \P(\w^i),$$ which contracts the extremal ray of the Mori cone
\[
R_1:=\left(\cap_{1\leq i \leq k+1 } H_{i}^{\perp}\right) \cap \overline{\text{NE}}(Y).
\] 
Indeed, consider the Stein factorisation 
\[
\begin{tikzcd}
Y \arrow[r,"\overline{\pi_1}"] &  Z_1 \arrow[r,"p_1"] & X_1
\end{tikzcd}
\]
of $\pi_{1}$. The exceptional locus of $\overline{\pi_1}$ is the set of points of $Y$ such that $\{f_0=f_1=f_2=0\}\cap Y$, and has codimension $2$ in $Y$, by the generality of $Y$. Whence $\overline{\pi_1}$ is a small contraction, and, by the triviality of $K_Y$, it is flopping. The morphism $p_1$ is a double cover of $X_1$, ramifying along the divisor $\{f_1^2-4f_0f_2=0\}\subset X_1$, and induces an automorphism of $Z_1$ of order two, exchanging the two sheets of the cover. This automorphism induces a birational involution of $Y$, which we call $\iota_1$. We claim that $\iota_1$ is the $-H_1$-flop of the flopping contraction $\overline{\pi}_1$. Indeed, we observe that $H_1+\iota_{1,*}(H_1)=\pi_1^*(\pi_{1,*}(H_1))$, hence $(H_1+\iota_{1,*}(H_1))\cdot C=0$, for any irreducible curve $C$ contracted by $\pi_1$. Moreover, $H_1\cdot C=1$ for any curve $C$ contracted by $\overline{\pi_1}$.
Then $\iota_{1,*}(H_1)$ is a $-\overline{\pi}_1$-ample divisor on $Y$, and is sent to $H_1$ via $\iota_1$, which is a $\overline{\pi}_1$-ample divisor. Whence $\iota_1$ is the $-H_1$-flop of $\overline{\pi_1}$. Then, up to automorphisms, $Y$ admits one and only one small $\Q$-factorial modification.

To conclude, we prove that $\mathrm{Mov}(Y)=\mathrm{Nef}(Y)\cup \iota_1^*(\mathrm{Nef}(Y))$. The inclusion $\supset$ is obvious. We are left to show $\subset$. Let $D$ be any Cartier divisor whose class lies in $\mathrm{Mov}(Y)$. If $D$ is nef, we are done, otherwise, we may assume that the pair $(Y,D)$ is klt and run an MMP for it. The first step of this MMP must be the involution $\iota_1$. Since $\iota_1$ has order $2$, the pair $(Y,\iota_{1,*}(D))$ is a (good) minimal model of $(Y,D)$. This proves the inclusion $\subset$.

\underline{Proof of (3)}. Since $m\geq 2$, arguing as in item (2), for any projection 
\[
\pi_{l} \colon Y \to X_l:= (\P^1)^{n-1} \times \left(\prod_{i=1}^k \P(\w^i)\right),
\] 
obtained by forgetting the $\P^1$ in position $l$, we can construct a birational involution $\iota_{l}$. Any two different such involutions do not commute, and hence generate an infinite subgroup of $\mathrm{Bir}(Y)$. Then, modulo automorphisms, $Y$ has infinitely many small $\Q$-factorial modifications, and this implies that $Y$ is not a Mori dream space.
\end{proofD}

We now prove Theorem \ref{thmC}.

\begin{proofC}
Let $D$ be any big and movable $\R$-divisor on $Y$. Up to picking a small positive multiple $\epsilon D$ of $D$, the pair $(Y,D)$ has klt singularities, and so has a minimal model $(Y',D')$, by \cite[Theorem 1.2]{BCHM10}. To be more precise, by running an MMP with scaling, we find a sequence of $D$-flops
\begin{equation}\label{flops}
(Y,D)=(Y_0,D)\dashrightarrow (Y_1,D_1) \dashrightarrow\cdots \dashrightarrow (Y_n,D_n)=(Y',D'),
\end{equation}
with $D'$ nef, and $Y'$ $\Q$-factorial and terminal. We claim that $Y'$ is isomorphic to $Y$. To this end, we show that the flop of any flopping contraction is to a variety isomorphic to $Y$. Since by Proposition \ref{propCY} $\mathrm{Nef}(Y)=\mathrm{Nef}(X)_{|Y}$, every flopping contraction of $Y$ must arise from a projection
\[
\pi_{l} \colon Y \to X_l:=(\P^1)^{n-1} \times \left(\prod_{i=1}^k \P^{m_i}(\w^i)\right),
\] 
obtained by forgetting the factor $\P^1$ in position $l$, for some $l\in\{1,\dots,n\}$. Indeed, all the other extremal contractions of $Y$ are onto projective varieties of dimension strictly smaller than $\text{dim}(Y)$.  Arguing as in Proposition \ref{propCY}, item (2), we see that every $\pi_l$ induces a flopping contraction. In particular, if
\[
\begin{tikzcd}
Y \arrow[r,"\overline{\pi_l}"] &  Z_l \arrow[r,"p_l"] & X_l
\end{tikzcd}
\]
is the Stein factorisation of $\pi_{l}$, the morphism $\overline{\pi_l} \colon Y \to Z_l$ contracts the extremal ray 
\[
R_l:=\left(\cap_{\substack{1\leq i \leq n+k \\ i\neq l}} H_{i}^{\perp}\right) \cap \overline{\text{NE}}(Y)
\] 
of the Mori cone of $Y$. The morphism $p_l$ is a double cover of $X_l$, ramifying along the divisor $\{f_1^2-4f_0f_2=0\}\subset X_l$, and induces an automorphism of $Z_l$ of order two, exchanging the two sheets of the cover. This automorphism induces a birational involution of $Y$, which we call $\iota_l$. Arguing as in Proposition \ref{propCY}, item (2), one sees that $\iota_l$ is the $-H_l$-flop of the flopping contraction $\overline{\pi}_l$. 
\vspace{0.1cm}

It follows that all the varieties appearing in sequence  (\ref{flops}) are isomorphic to $Y$, and the birational maps in there must be some of the involutions $\iota_l$, for certain values of $l$. This implies that $[D]$ is the image of some nef class via some element of $\mathrm{Bir}(Y)$, whence $\text{Bir}(Y)\cdot \text{Nef}(Y) \supset \text{int}(\text{Mov}(Y))$. Moreover $\mathrm{Nef}(Y)\subset \overline{\mathrm{Mov}}^+(Y)$, because $\mathrm{Nef}(Y)$ is spanned by semiample divisors. Since $X$ has terminal singularities, by \cite[Lemma 5.17]{KM98} the variety $Y$ is terminal. Then any element of $\mathrm{Bir}(Y)$ is a small $\Q$-factorial modification, by \cite[Corollary 3.54]{KM98}, and more precisely a composition of flops, by \cite[Theorem 1]{Kawamata08}. It follows that $\mathrm{Bir}(Y)$ is generated by the involutions $\iota_l$ and $\mathrm{Aut}(Y)$. Note that $\mathrm{Bir}(Y)$ acts on $\mathrm{Pic}(Y)$. Indeed, for any $1\leq l\leq n$, the involution $\iota_l$ fixes $H_i$, for $i\neq l$, and sends $H_l$ to $$\left[\sum_{i\neq l} 2H_i+\sum_{l+1\leq i\leq l+k}\left(\sum_jw_j^{i}\right)H_i\right]-H_l.$$ By applying \cite[Proposition-Definition 4.1]{Looijenga14}, we obtain $\mathrm{Bir}(Y)\cdot \mathrm{Nef}(Y)=\overline{\mathrm{Mov}}^+(Y)$. Since any extremal generating divisor for $\mathrm{Nef}(Y)$ is semiample, we conclude that $\overline{\text{Mov}}^+(Y)=\mathrm{Mov}(Y)$, because being movable is preserved under small $\Q$-factorial modifications, and $\mathrm{Mov}(Y)\subset \mathrm{Bir}(Y)\cdot \mathrm{Nef}(Y)$.
Now we use \cite[Application 4.14]{Looijenga14} (see also \cite[Lemma 3.5]{LZ22}) to deduce the existence of a fundamental domain $\Pi\subset \overline{\text{Mov}}^+(Y)$ for the action of $\text{Bir}(Y)$ on $\overline{\text{Mov}}^+(Y)=\mathrm{Mov}(Y)$.
\end{proofC}

\printbibliography
\end{document}